\documentclass{article}
\usepackage[utf8]{inputenc}
\usepackage[margin=1.2in]{geometry}
\usepackage[colorlinks]{hyperref}
\usepackage{microtype}
\usepackage{mathtools, amsthm, thmtools, amsfonts, amssymb,amsmath}
\usepackage{thm-restate}
\usepackage{bm}

\usepackage{color}
\usepackage{tikz}

\DeclareMathOperator{\Int}{int}
\DeclareMathOperator*{\argmax}{arg\,max}
\DeclareMathOperator*{\argmin}{arg\,min}

\newcommand{\Chat}{\hat{\mathbb{C}}}

\declaretheorem[parent=section]{theorem}
\declaretheorem[sibling=theorem]{lemma}

\declaretheorem[sibling=theorem, style=definition]{definition}
\declaretheorem[sibling=theorem, style=definition]{corollary}
\declaretheorem[sibling=theorem, style=remark]{remark}

\declaretheorem[sibling=theorem, style=definition]{question}

\newcommand{\B}[1]{B_{#1}}

\definecolor{c0}{RGB}{255,0,0}
\definecolor{c1}{RGB}{204,255,0}
\definecolor{c2}{RGB}{0,255,102}
\definecolor{c3}{RGB}{0,102,255}
\definecolor{c4}{RGB}{204,0,255}

\title{Optimal zero-free regions for the independence polynomial of bounded degree hypergraphs}
\author{Ferenc Bencs\quad\quad\quad Pjotr Buys}

\begin{document}
\maketitle

\begin{abstract}
In this paper we  investigate the distribution of zeros of the independence polynomial of  hypergraphs of maximum degree $\Delta$. For graphs the largest zero-free disk around zero was described by Shearer as having radius $\lambda_s(\Delta)=(\Delta-1)^{\Delta-1}/\Delta^\Delta$. Recently it was shown by Galvin et al. that for hypergraphs the disk of radius $\lambda_s(\Delta+1)$ is zero-free; however, it was conjectured that the actual truth should be $\lambda_s(\Delta)$. We show that this is indeed the case. We also show that there exists an open region around the interval $[0,(\Delta-1)^{\Delta-1}/(\Delta-2)^\Delta)$ that is zero-free for hypergraphs of maximum degree $\Delta$, which extends the result of Peters and Regts from graphs to hypergraphs.
Finally, we determine the radius of the largest zero-free disk for the family of bounded degree $k$-uniform linear hypertrees in terms of $k$ and $\Delta$.
\end{abstract}


\section{Introduction}
Let $\mathcal{H} = (V,E)$ be a hypergraph, where $V$ denotes a finite set of vertices and $E$ is a set of nonempty subsets of $V$. An independent set is a subset of the vertices $U \subseteq V$ such that $e \not\subseteq U$ for any $e \in E$. We let $\mathcal{I}(\mathcal{H})$ denote the set of independent sets of $\mathcal{H}$. If we assign to every vertex $v \in V$ a complex variable $\bm{\lambda}_v$, we can define the multivariate independence polynomial of $\mathcal{H}$ as 
\[
    Z(\mathcal{H}; \bm{\lambda}) = \sum_{U \in \mathcal{I}(\mathcal{H})} \prod_{v \in U} \bm{\lambda}_v.
\]
We let $Z(\mathcal{H}; \lambda)$ denote the univariate specialization, i.e. where all $\bm{\lambda}_v$ are set to $\lambda$.

The location of the zeros of the independence polynomial in the complex plane within the class of bounded degree graphs has been a topic gleaning a lot of interest 
\cite{leake2019generalizations,chudnovsky2007roots,brown2018stability}. Recently this interest has been extended to the study of bounded degree hypergraphs \cite{PerkinsHypergraphs}. The following two theorems contain our main contribution.

The first result concerns the largest zero-free disk around zero. Let $\lambda_s(\Delta)=\frac{(\Delta-1)^{\Delta-1}}{\Delta^\Delta}$. 

\begin{restatable}{theorem}{theoremShearer}
    \label{thm: Shearer all hypergraphs}
    Let $\Delta \geq 2$.
    For any hypergraph $\mathcal{H}$ with maximum degree at most $\Delta$ and $\bm{\lambda} \in \mathbb{C}^{V(\mathcal{H})}$ with $|\bm{\lambda}_u| \leq \lambda_s(\Delta)$ for all $u\in V(\mathcal{H})$ we have $Z(\mathcal{H};\bm\lambda)\neq 0$.
\end{restatable}

For graphs this bound was proved by Shearer \cite{Shearer1985}; see also \cite{ScottSokal}. For hypergraphs our result improves the bound from \cite{PerkinsHypergraphs} from $\lambda_s(\Delta+1)$ to $\lambda_s(\Delta)$. This bound is the best possible since it is proved in \cite{Shearer1985, ScottSokal} that zeros of truncated ($\Delta-1$)-ary trees accumulate on $-\lambda_s(\Delta)$.

The second result concerns the largest possible open zero-free region around the positive real axis. Let $\lambda_c(\Delta)=\frac{(\Delta-1)^{\Delta-1}}{(\Delta-2)^\Delta}$.

\begin{restatable}{theorem}{theoremSokal}
    \label{thm: hypergraph Sokal region}
    Let $\Delta\ge 3$. There exists an open neighborhood $U$ of the interval $(0,\lambda_c(\Delta))$ such that for any hypergraph $\mathcal{H}$ of maximum degree $\Delta$ and $\lambda\in U$ we have $Z(\mathcal{H}; \lambda)\neq 0$. 
\end{restatable}

For graphs this was conjectured by Sokal \cite{Sokal2001} and subsequently proved by Peters and Regts \cite{PetersRegtsSokal}. It is also proved in \cite{PetersRegtsSokal} that zeros of truncated ($\Delta-1$)-ary trees accumulate on $\lambda_c(\Delta)$. In fact, it is proved in \cite{BBGPR}, using results from \cite{GoldbergEtAl2018}, that zeros of graphs of maximum degree at most $\Delta$ accumulate on any real parameter in $(-\infty,-\lambda_s(\Delta)] \cup [\lambda_c(\Delta),\infty)$. Therefore, if we let $\mathcal{U}_{\Delta,2} \subseteq \mathbb{C}$ denote the maximal open zero-free set for graphs with maximum degree at most $\Delta$ and we let $\mathcal{U}_{\Delta,\geq 2} \subseteq \mathbb{C}$ denote the analogous set for hypergraphs with maximum degree at most $\Delta$, we can conclude the following corollary.

\begin{corollary}
    Let $\Delta \geq 3$. We have
    \[
        \mathcal{U}_{\Delta,2} \cap \mathbb{R} = \mathcal{U}_{\Delta,\geq 2} \cap \mathbb{R} = \left(-\lambda_s(\Delta), \lambda_c(\Delta)\right).
    \]
\end{corollary}

In \cite[Conjecture 6]{PerkinsHypergraphs} it is conjectured that the connected component of $\mathcal{U}_{\Delta,2}$ containing $0$ is equal to the connected component of $\mathcal{U}_{\Delta,\geq 2}$ containing $0$. This conjecture is false for $\Delta = 1,2$ because graphs of maximum degree at most two are real-rooted, while the independence polynomial of the hypergraph consisting of a single hyperedge with three vertices is $3\lambda^2 + 3\lambda + 1$, which has non-real roots. The conjecture is also false for $\Delta = 3$; see Corollary~\ref{cor: U3 is not zero-free}. We will however show that the conjecture is true in the $\Delta\to\infty$ limit.
\begin{theorem}
    \label{thm: limit theorem}
The rescaled zero-free regions converge to the same limit object in terms of Hausdorff distance\footnote{A sequence of sets $A_n$ converges to $A$ if for every closed $K_1 \subseteq \Int(A)$ and open $K_2 \supseteq \overline{A}$ for $n$ sufficiently large $K_1 \subseteq A_n \subseteq K_2$.}, i.e.
\[
    \lim_{\Delta \to \infty} \Delta \cdot \mathcal{U}_{\Delta,2} =  \lim_{\Delta \to \infty} \Delta \cdot \mathcal{U}_{\Delta,\geq 2}
\]
\end{theorem}
A priori it is not clear that either limit is well-defined. In \cite{BencsBuys2021} the two authors of the present paper together with Han Peters investigated the left-hand limit and showed that it does indeed converge to a non-trivial limit set. 


\subsection{Uniform linear hypertrees}
A hypergraph is called linear if for any pair of distinct edges $e_1,e_2$ we have $|e_1 \cap e_2| \leq 1$. A path from vertex $v_1$ to vertex $v_2$ in a linear hypergraph is a tuple $(v_1 = u_1, e_1, u_2, e_2, \dots, e_{n-1},u_n= v_2)$, where $u_1, \dots, u_n$ is a sequence of distinct vertices and $e_1, \dots, e_{n-1}$ is a sequence of distinct edges such that $\{u_i,u_{i+1}\} \subseteq e_i$ for all $i=1, \dots, n-1$. A linear hypergraph is called a linear hypertree if there is a unique path between any pair of vertices. A hypergraph is called $k$-uniform if every hyperedge has size $k$.

In \cite{PerkinsHypergraphs} the largest possible zero-free disk for bounded degree uniform linear hypertrees is investigated. They give an explicit radius $r_{d,b}$ such that for any $b+1$ uniform linear hypertree $\mathcal{T}$ with maximum degree at most $d+1$ and fugacities $\bm{\lambda}$ with $|\bm{\lambda}_v| \leq r_{d,b}$ it is the case that $Z(\mathcal{T};\bm{\lambda}) \neq 0$; see \cite[Theorem 4]{PerkinsHypergraphs}. As $d \to \infty$ one has that $r_{d,b} = e^{-1}d^{-1/b} + \mathcal{O}(d^{-2/b})$. They also show that a disk of radius $\mathcal{O}((\log(d)/d)^{1/b})$ is not zero-free; see \cite[Proposition 5]{PerkinsHypergraphs}. In this paper we determine the exact maximal zero-free disk for uniform bounded degree linear hypertrees.

\begin{restatable}{theorem}{theoremUniformHypertrees}
    \label{thm: Zero free linear (k)-uniform hypertrees}
    Let $d \geq 2$, $b \geq 1$ and define the function $f_\lambda(z)=f_{\lambda,b,d}(z)$ as
    \[
        f_{\lambda}(z)=\lambda\left(1-\left(\frac{z}{1+z}\right)^b\right)^d.
    \]
    Let $\rho_{d,b}$ be the largest $\rho\ge 0$, such that there exists a disk $\B{R}$ of radius $R$ around $0$ that is mapped into itself by $f_\rho$. To be precise,
    \[
        \rho_{d,b}=\max\{\rho\ge 0~|~\text{there exists an } R>0\text{ such that }f_\rho(\B{R})\subseteq \B{R}\}.
    \]
    
    \begin{itemize}
        \item[{\bf Zero-freeness}]
            If $\mathcal{T}$ is a $(b+1)$-uniform linear hypertree  with maximum degree at most $d+1$ and ${\bm \lambda} \in \mathbb{C}^{V(\mathcal{T})}$ is such that $|{\bm\lambda}_v| \leq \rho_{d,b}$ for all $v\in V(\mathcal{T})$, then  $Z(\mathcal{T}; \bm \lambda) \neq 0$.
    \item[{\bf Maximality}]
        There is a sequence of $(b+1)$-uniform linear hypertrees $\{\mathcal{T}_n\}$ with maximum degree at most $d+1$ and a sequence of parameters $\{\lambda_n\}$ such that $|\lambda_n|$ converges to $\rho_{d,b}$ and $Z(\mathcal{T}_n;\lambda_n) = 0$ for all $n$.

    \item[{\bf Formula}]
        \[
            \rho_{d, b} = \min\left\{\left|(1-w^{b})^{-d}\frac{w}{(1-w)}\right|: w\in \mathbb{C} \text{ with } -bd\frac{w^{b}(1-w)}{(1-w^b)} = 1 \right\}.
        \]
    \item[{\bf Asymptotics}]
        For fixed $b$ 
    \[
        \rho_{d,b} = (ebd)^{-1/b} + \mathcal{O}(d^{-2/b})
    \]
    as $d \to \infty$.
    \end{itemize}
\end{restatable}

Contrary to the proof of Theorem~\ref{thm: Shearer all hypergraphs}, where we use the explicit formula for $\lambda_s(\Delta)$, we will not prove Theorem~\ref{thm: Zero free linear (k)-uniform hypertrees} using the formula. Instead, we will show that if a natural sufficient zero-freeness condition, namely forward invariance of the map $f_\lambda$, is not met this implies that there is a $\lambda'$ of the same magnitude as $\lambda$ for which $f_{\lambda'}$ has neutral fixed point. This will allow us to prove that zeros accumulate on $\lambda'$ and to give an explicit formula for $\lambda'$ and its magnitude.

\begin{question}
    A few days before appearance of the present paper it was shown by Zhang \cite{zhang2023note} that for each $k \geq 3$ there exists a sequence of $k$-uniform linear hypergraphs with zeros of magnitude  $\mathcal{O}(\log(\Delta)/\Delta)$. It follows that in general the disk of radius $\rho_{d,b}$ does not remain zero-free if one moves from uniform linear hypertrees to uniform linear hypergraphs. Describing the true asymptotics of the maximal zero-free disk for uniform bounded degree linear hypergraphs remains an open problem. The Shearer disk \cite{Shearer1985} shows that it is at least of size $\lambda_s(\Delta) = (e\Delta)^{-1} + \mathcal{O}(\Delta^{-2})$.
\end{question}

\subsection{Algorithmic consequence}

Let $\mathcal{F}_{\Delta,k}$ be the family of hypergraph of maximum degree at most $\Delta$ and of hyperedge-size at most $k$. In the paper \cite{PerkinsHypergraphs} an FPTAS\footnote{An FPTAS for a graph polynomial $Z_G$ is an algorithm that on an input $G$ and $\epsilon>0$ outputs an $\hat Z$ in polynomial time $|V(G)|$ and $1/\epsilon$ such that it is an $\epsilon$-relative approximation of $Z_G$, i.e. $|Z_G-\hat Z|\le \epsilon Z_G$.} was established for computing $Z(\mathcal{H};\lambda)$ for the family $\mathcal{F}_{\Delta,k}$ and $|\lambda|<\lambda_s(\Delta+1)$.
The proof of \cite[Theorem~10]{PerkinsHypergraphs} is based on 3 ingredients: Barvinok's interpolation method~\cite{barvinok2016combinatorics}, the work of Liu, Sinclair and Srivastava~\cite{liu2019ising} and Patel and Regts \cite{PatelRegts2017}, and the zero-free region given in \cite[Theorem~1]{PerkinsHypergraphs}. In generality they proved the following.
\begin{theorem}[\cite{PerkinsHypergraphs}]
Let $\Delta\ge 2$ and $k\ge 2$. Assume that there is a connected open set $0\in U\subseteq \mathbb C$, such that for any $\mathcal{H}\in\mathcal{F}_{\Delta,k}$ and $\lambda\in U$ we have $Z(\mathcal{H};\lambda)\neq 0$. 
Then for any $\lambda \in U$ there is an algorithm of running time $(n/\epsilon)^{O_{k,\Delta}(1)}$ such that it computes an $\epsilon$-relative approximation to $Z(\mathcal{H};\lambda)$ for any hypergraph $\mathcal{H}\in\mathcal{F}_{\Delta,k}$ on $n$ vertices.
\end{theorem}

There are several works on the complexity of counting independent sets in hypergraphs, i.e. calculating $Z(\mathcal{H};1)$; see \cite{bordewich2008path,bezakova2019approximation,liu2014fptas,hermon2019rapid,PerkinsHypergraphs}. By combining the previous theorem with Theorem~\ref{thm: hypergraph Sokal region} and Theorem~\ref{thm: Shearer all hypergraphs} we obtain the following corollary.

\begin{corollary}
    Let $\Delta\ge 2$ and $k\ge 2$. Let $\lambda\in\mathbb C$, such that either $|\lambda|<\lambda_s(\Delta)$ or $0<\lambda<\lambda_c(\Delta)$. Then there is an algorithm of running time $(n/\epsilon)^{O_{k,\Delta}(1)}$ such that it computes an $\epsilon$-relative approximation to $Z(\mathcal{H};\lambda)$ for any hypergraph $\mathcal{H}\in\mathcal{F}_{\Delta, k}$ on $n$ vertices.
\end{corollary}

This result recovers the result of Liu and Lu~\cite{liu2014fptas}, where they give an FPTAS for the number of independent sets (i.e. $\lambda=1$) for the class of hypergraphs of maximum degree at most $5$. More generally, this corollary extends the well-known result of the existence of an FPTAS for $\lambda \in (0, \lambda_c(\Delta))$ by Weitz \cite{weitz2006counting} from bounded degree graphs to bounded degree hypergraphs. It is known that for $\lambda < -\lambda_s(\Delta)$ and $\lambda > \lambda_c(\Delta)$ approximating the independent set polynomial for bounded degree graphs is NP-hard; see \cite{GoldbergEtAl2018} and \cite{Slysun} for the two regimes. We thus show that the computational dichotomy for real $\lambda$ that was known for bounded degree graphs extends to bounded degree hypergraphs.


A different consequence of our zero-freeness result concerns the mixing time of a Markov chain known as Glauber dynamics. For positive real $\lambda$ Glauber dynamics are used to approximately sample from the Gibbs distribution, i.e. the distribution on $\mathcal{I}(\mathcal{H})$ where the probability of sampling $U$ is proportional to $\lambda^{|U|}$. Anari et al. \cite{AnariSpectral} proposed a sufficient condition for the rapid mixing of the Glauber dynamics known as spectral independence, which was investigated for special Gibbs distributions \cite{chen2021optimal,blanca2022mixing}. Furthermore, it was shown in \cite[Theorem~11]{chen2022spectral} that zero-freeness in a neighborhood of $\lambda$ implies spectral independence for a large class of models on bounded degree (hyper)graphs, including the hard-core model on hypergraphs. These results together with Theorem~\ref{thm: hypergraph Sokal region} prove rapid mixing of the Glauber dynamics for bounded degree hypergraphs for $\lambda \in (0,\lambda_c(\Delta))$. 

\subsection{Overview of the paper}
The proofs of the main theorems rely on the dynamical behaviour of ratios associated to the independence polynomials. For a hypergraph $\mathcal{H}$ and a vertex $v\in V(\mathcal{H})$ we can define the ratio as
\[
    R_v(\mathcal{H},\lambda)=\frac{Z^\textrm{in}_{v}(\mathcal{H},\lambda)}{Z^\textrm{out}_{v}(\mathcal{H},\lambda)},
\]
where $Z_v^\textrm{in}(\mathcal{H},\lambda)$ is the sum of the terms of the partition function of $Z(\mathcal{H};\lambda)$ with $v$ contained in the independent set and $Z^\textrm{out}_v(\mathcal{H};\lambda)=Z(\mathcal{H};\lambda)-Z_v^{\textrm{in}}(\mathcal{H};\lambda)$.

In Section~\ref{sec: pre} we build on the reduction of Weitz \cite{weitz2006counting} to reduce the analysis of the ratios of hypergraphs to the analysis of ratios of linear hypertrees. It turns out that the ratios of linear hypertrees can be calculated recursively from smaller subhypertrees using appropriate multivariate functions.

In Section~\ref{sec: zerofreeness} we use this reduction to show that establishing forward invariant sets in the complex plane for these (multivariate) functions is a sufficient condition for zero-freeness. We gather results of this form for both the univariate and the multivariate case. This section also contains the proof of Theorem~\ref{thm: Shearer all hypergraphs}.

In Section~\ref{sec: Sokal} we show that we can satisfy the sufficient conditions for zero-freeness in the univariate case laid out in Lemma~\ref{lem: properties of A} for $\lambda \in (0, \lambda_c(\Delta))$. This leads to a proof of Theorem~\ref{thm: hypergraph Sokal region}.

In Section~\ref{sec: limit region} we use a characterization of the limit set $\mathcal{U}  = \lim_{\Delta \to \infty} \Delta \cdot \mathcal{U}_{\Delta,2}$ given in \cite{BencsBuys2021}. This allows us to show that for $\Lambda \in \mathcal{U}$ and $\Delta$ large enough there exists a region $A_\Delta$ satisfying the conditions of Lemma~\ref{lem: properties of A} for the parameter $\Lambda/\Delta$. This leads to a proof of Theorem~\ref{thm: limit theorem}.

In Section~\ref{sec: disk for linear hypertrees} we use Lemma~\ref{lem: zero-free disk k-uniform hypertrees} to establish a zero-free disk for $k$-uniform linear hypertrees. In fact, we show that the smallest disk to which we cannot apply Lemma~\ref{lem: zero-free disk k-uniform hypertrees} implies that the map $f_\lambda$ given in equation~{(\ref{eq: flambda hypertrees})} has a neutral fixed point. This will allow us both to prove that zeros accumulate on $\lambda$, following the method in \cite{buys_cayleytrees}, and to establish a formula for $\lambda$, proving Theorem~\ref{thm: Zero free linear (k)-uniform hypertrees}. In this section we rely on methods from the field of complex dynamics; in particular, we use the theory of normal families and holomorphic motion.

\section{Preliminaries}\label{sec: pre}
\subsection{Independence polynomial of hypergraphs}
For a hypergrpah $\mathcal{H}$ and its vertex $v\in V(\mathcal{H})$ let us define the ratio of $R_v(\mathcal{H},\bm{\lambda})$ as the multivariate rational function
\[
    R_v(\mathcal{H};\bm{\lambda})=\frac{Z_v^{\textrm{in}}(\mathcal{H};\bm{\lambda})}{Z_v^{\textrm{out}}(\mathcal{H};{\bm\lambda})},
\]
where $Z_v^{\textrm{in}}(\mathcal{H};{\bm\lambda})$ (resp. $Z_v^{\textrm{out}}(\mathcal{H};{\bm\lambda})$) is the sum of those terms of the partition function of $Z(\mathcal{H};\bm{\lambda})$ where $v$ is in (resp. out) the independent set of $U$. To be precise,
\[
    Z_v^{\textrm{in}}(\mathcal{H};{\bm\lambda})=\sum_{v\in U\in \mathcal{I}(\mathcal{H})}\prod_{u\in I} \lambda_u
\]
and $Z_v^{\textrm{out}}(\mathcal{H};\bm{\lambda})=Z(\mathcal{H};\bm{\lambda})-Z_v^\textrm{in}(\mathcal{H};\bm{\lambda})$. 
In general, for $S\subseteq V$ we can define
\[
    Z_S^{\textrm{in}}(\mathcal{H};\bm\lambda)=\sum_{S\subseteq U\in\mathcal{I}(\mathcal{H})}\prod_{u\in U}\lambda_u.
\]

Before establishing basic recursions we need the following notations and conventions for the rest of this section.

\begin{definition} Let $\mathcal{H}$ be a hypergraph on the vertex set $V$. We can then consider the following ways to obtain subhypergraphs of $\mathcal{H}$.
\begin{itemize}
\item If $S\subseteq V$, then $\mathcal{H}-S$ is the induced subhypergraph on $V\setminus S$. That is, a hypergraph with vertex set $V\setminus S$ and with hyperedges
\[
    \{e\in E ~|~ e\cap S=\emptyset \}.
\]
\item If $S\subseteq V$, then $\mathcal{H}\ominus S$ is the hypergraph obtained by removing all vertices of $S$ from each hyperedge. If the hypergaph obtains empty hyperedges, then we delete those. Formally, $\mathcal{H}\ominus S$ is a hypergraph on $V\setminus S$ with hyperedges
\[
    \{e\setminus S ~|~ e\in E \textrm{ s.t. } e\cap S\neq e \}.
\]
\item If $F\subseteq E$, then $\mathcal{H}-F$ is the hypergraph obtained by removing the hyperedges $F$ from the hypergraph. That is, a hypergraph with vertex set $V$ and with hyperedges
\[
    \{e\in E~|~ e\notin F\}.
\]
\end{itemize}
A hypergraph $\mathcal{H}'$ is a subhypergraph of $\mathcal{H}$ if $\mathcal{H}'$ can be obtained from $\mathcal{H}$ by applying any of the previous operations and/or their combinations. We will use the convention in this case that $Z(\mathcal{H}';\bm{\lambda})$ denotes $Z\left(\mathcal{H}';\bm{\lambda}\big|_{V(\mathcal{H}')}\right)$.
\end{definition}


Now let us collect basic relations following from the definition of independent sets. We refer to \cite{trinks2016survey} for their proofs.
\begin{lemma}\label{lem:basic_recursion}
Let $\mathcal{H}=(V,E)$ be a hypergraph  and $v\in V$ such that $\{v\}\notin E$. Then
\begin{enumerate}
\item $Z_v^\textrm{in}(\mathcal{H},\bm\lambda)=\lambda_v Z(\mathcal{H}\ominus v,\bm\lambda)$.
\item $Z_v^\textrm{out}(\mathcal{H},\bm\lambda)=Z(\mathcal{H}- E(v)-v,\bm\lambda)$, where $E(v)=\{e\in E ~|~ v\in e\}$.
\item $Z(\mathcal{H};\bm\lambda)=Z(\mathcal{H}-\{e\};\bm\lambda)-Z_e^\textrm{in}(\mathcal{H}-\left\{e\right\};\bm\lambda)$.
\end{enumerate}
\end{lemma}
A short observation of the previous lemma is that the independence polynomial of a hypergraphs is actually a product of independence polynomials of hypergraphs where no hyperedge is contained in an other, moreover each hyperedge has size at least 2.
\begin{lemma}\label{lemma: at least 2}
Let $\mathcal{H}=(V,E)$ be a hypergraph and let $V_1=\{v\in V~|~\{v\}\in E\}$ and $E'=\{e\in E~|~\exists f\in E \textrm{ s.t. } f\subset e\}\cup\{e\in E~|~|e|=1\}$. Then the hypergraph $\mathcal{H}-E'-V_1$ is disjoint union of hypergraphs of hyperedge-size at least 2 and there is no hyperedge that contains an other one, moreover
\[
    Z(\mathcal{H};\bm\lambda)=Z(\mathcal{H}-E'-V_1;\bm\lambda).
\]
\end{lemma}

The next lemma is the generalized form of the tree recursion of the independence polynomial of trees for hypertrees.

\begin{lemma}\label{lem:tree_recursion}
Let $\mathcal{T}=(V,E)$ be a linear hypertree and $v$ a vertex. Suppose that $E(v)=\{e_1,\dots,e_d\}\subseteq E$ is the set of the incident edges at $v$ and $e_i=\{v,v^{(i)}_1,\dots,v^{(i)}_{b_i}\}$ for $i=1,\dots,d$. Let $\mathcal{T}_{i,j}$ be the connected component of $v^{(i)}_j$ in $\mathcal{T}-v$. Then we have
\[
R_{v}(\mathcal{T};\bm\lambda)=\lambda_v\prod_{i=1}^d\left(1-\prod_{j=1}^{b_i}\frac{R_{v^{(i)}_j}(\mathcal{T}_{i,j};\bm\lambda)}{1+R_{v^{(i)}_j}(\mathcal{T}_{i,j};\bm\lambda)}\right).
\]
\end{lemma}
\begin{proof}
Let us recall that $Z(\mathcal{T};{\bm\lambda})=Z_v^\textrm{in}(\mathcal{T};{\bm\lambda})+Z_v^\textrm{out}(\mathcal{T};{\bm\lambda})$. 
If $\{v\}\in E(\mathcal{T})$, then $R_v(\mathcal{T};{\bm\lambda})=0$ by definition.

Thus for the rest let us assume that $\{v\}\notin E$.
Then by Lemma~\ref{lem:basic_recursion} we have
\[
    R_v(\mathcal{T};\bm\lambda)=\frac{Z_v^\textrm{in}(\mathcal{T};{\bm\lambda})}{Z_v^\textrm{out}(\mathcal{T};{\bm\lambda})}=\frac{\lambda_v Z(\mathcal{T}\ominus\{v\};{\bm\lambda}) }{Z(\mathcal{T}-E(v)-v;{\bm\lambda}) }
\]
Observe that the hypergraphs $\mathcal{T}_1=\mathcal{T}\ominus\{v\}$ and $\mathcal{T}_0=\mathcal{T}-E(v)-v$ are defined on the same vertex set and their edge set differs only in $E'=\{e\setminus\{v\} ~|~ e\in E(v)\}$. Since $\mathcal{T}$ is a linear hypertree each connected component of $\mathcal{T}_1$ and $\mathcal{T}_0$ is a linear hypertree. The hyperforest $\mathcal{T}_0$ is isomorphic to the disjoint union of the hypertrees $\mathcal{T}_{i,j}$ for $i=1,\dots,d$ and $j=1,\dots,b_i$. Thus $\mathcal{T}_1$ has exactly $d$ connected components, i.e. for every $i=1,\dots,d$ we can obtain the connected component as the disjoint union of $\mathcal{T}_{i,j}$ for $j=1,\dots,b_i$ and adding the edge $e_i'=e_i\setminus\{v\}$. Let us call these connected components $\mathcal{T}_i$.

Thus we have that
\[
R_v(\mathcal{T},\bm\lambda)=\lambda_v\prod_{i=1}^d\frac{Z(\mathcal{T}_i,\bm\lambda)}{Z(\mathcal{T}_i-\{e_i'\};\bm\lambda)}= \lambda_v\prod_{i=1}^d\frac{Z(\mathcal{T}_i-\{e_i'\},\bm\lambda)-Z_{e_i'}^\textrm{in}(\mathcal{T}_i-\{e_i\},\bm\lambda)}{Z(\mathcal{T}_i-\{e_i'\};\bm\lambda)}
\]
To finish the proof let us observe that for each $i=1,\dots,d$ we have 
\begin{align*}
\frac{Z(\mathcal{T}_i-\{e_i'\},\bm\lambda)-Z_{e_i'}^\textrm{in}(\mathcal{T}_i-\{e_i\},\bm\lambda)}{Z(\mathcal{T}_i-\{e_i'\};\bm\lambda)}&=1-\frac{Z_{e_i'}^\textrm{in}(\mathcal{T}_i-\{e_i\},\bm\lambda)}{Z(\mathcal{T}_i-\{e_i'\};\bm\lambda)}\\&=1-\frac{\prod_{j=1}^{b_i} Z^\textrm{in}_{v_i^{(j)}}(\mathcal{T}_{i,j};\bm\lambda)}{\prod_{j=1}^{b_i} Z(\mathcal{T}_{i,j};\bm\lambda)}\\&=1-\prod_{j=1}^{b_i}\frac{R_{v^{(i)}_j}(\mathcal{T}_{i,j};\bm\lambda)}{1+R_{v^{(i)}_j}(\mathcal{T}_{i,j};\bm\lambda)}.
\end{align*}
\end{proof}

\subsection{Reduction to hypertrees}

In the next definition we propose a hypergraph version of Weitz self-avoiding path tree \cite{weitz2006counting} that allows us to reduce the problem of understanding zeros of hypergraphs to linear hypertrees. This idea was implicitly used in \cite{liu2014fptas}.

\begin{definition} [Weitz-hypertree]\label{def:weitz} 
    Let $\mathcal{H}$ be a hypergraph and let $v\in V(\mathcal{H})$ be fixed. Then let us define recursively the following rooted hypertree $(\mathcal{T},r)$ and labeling $\pi:V(\mathcal{T})\to V(\mathcal{H})$.  If $\mathcal{H}$ is a linear hypertree, then $\mathcal{T}=\mathcal{H}$, $r=v$ and $\pi= \textrm{id}_{V(\mathcal{H})}$. Otherwise, let $\{e_1,\dots,e_d\}\in E(\mathcal{H})$ be the incident edges of $H$ at $v$. For each $i=1,\dots,d$ let $e_i=\{v,v_1^{(i)},\dots,v_{b_i}^{(i)}\}\subseteq V(\mathcal{H})$ and for each $j=1,\dots,b_i$ let $(\mathcal{T}_j^{(i)},r_j^{(i)})$ be the rooted hypertree obtained for the connected component of $v_{j}^{(i)}\in V(\mathcal{H}_j^{(i)})$ in the hypergraph $\mathcal{H}_{j}^{(i)}=\mathcal{H}-\{e_1,\dots,e_i\}\ominus\{v,v_1^{(i)},\dots,v_{j-1}^{(i)}\}$. Let $\pi_j^{(i)}$ be the labelings of the corresponding rooted hypertrees.

    Then $\mathcal{T}$ is the hypergraph obtained as follows: take disjoint union of the hypertrees $\mathcal{T}_j^{(i)}$, then add a new vertex $r$ and add the hyperedges $\{r,r_{1}^{(i)},\dots,r_{b_i}^{(i)}\}$ for every $i=1,\dots,d$. Let $\pi:V(\mathcal{T})\to V(\mathcal{H})$ be defined as
    \[
        \pi(u)=\left\{\begin{array}{cl}
                            v & \textrm{if $u=r$}\\
                            \pi_j^{(i)}(u) & \textrm{if $u\in V(\mathcal{T}_j^{(i)})$}
                        \end{array}\right.
    \]

    The construction comes with a natural labeling of $V(\mathcal{T})$ with the vertices of $\mathcal{H}$, which we will denote by $\pi$. 
\end{definition}

\begin{figure}[h]
    \centering
    \includegraphics{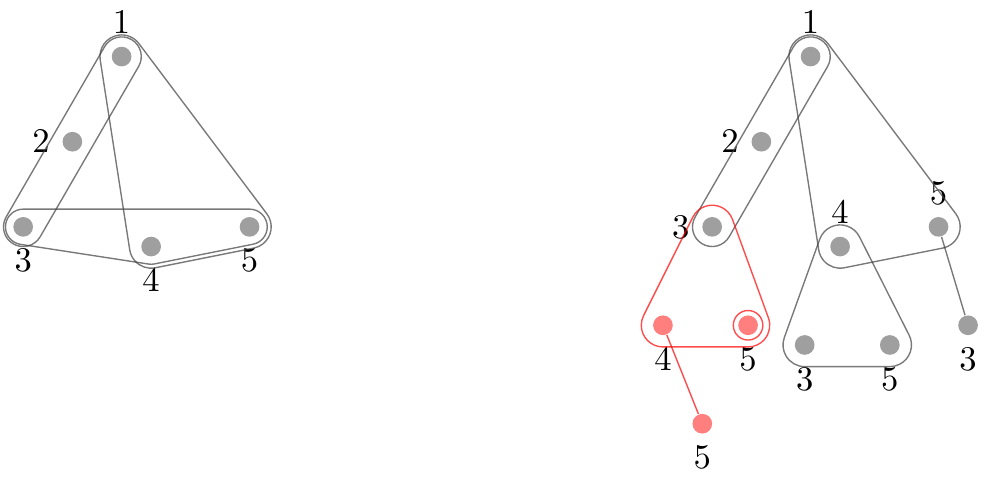}
    \caption{An example for Definition~\ref{def:weitz}, where we labeled the vertices of the linear hypertree (on the right) using the natural labeling $\pi$. The hypertree obtained by deleting the red vertices and red edges is the tree $\mathcal{T}'$ that is given by Remark~\ref{rem: weitz ratio} (and Remark~\ref{rem: weitz divisibility}). }
    \label{fig:weitz}
\end{figure}

\begin{theorem}\label{thm:ratio}
Let $\mathcal{H}$ be a hypergraph on $n$ vertices with degree at most $\Delta$ and edge-size at most $b+1$. Then for any $v\in V(\mathcal{H})$ the linear hypertree $\mathcal{T}$ and vertex $r$ of $\mathcal{T}$ defined by Definition~\ref{def:weitz} we have that
\[
    \frac{Z^\textrm{in}_{v}(\mathcal{H};\bm\lambda)}{Z^\textrm{out}_{v}(\mathcal{H};\bm\lambda)}=\frac{Z^\textrm{in}_{r}(\mathcal{T},\bm\lambda')}{Z_{r}^\textrm{out}(\mathcal{T},\bm\lambda')}
\]
as rational functions, where for any $u\in V(T)$ the variable $\lambda'_u=\lambda_{\pi(u)}$.

\end{theorem}
\begin{proof}
Proof by induction on the number of vertices of $\mathcal{H}$.

First of all we may assume that $\mathcal{H}$ is connected, since otherwise we could take $\mathcal{H}'$ to be the connected component of $v$ for which the ratio is the same function allowing us to conclude the theorem by induction.

If $|V(\mathcal{H})|=1$ or $|E(\mathcal{H})|\le 1$ the theorem holds because $\mathcal{H}$ is a hypertree and thus isomorphic to $\mathcal{T}$.

Now let us assume that for any connected hypergraph on $n\ge 1$ vertices the statement holds. Let $\mathcal{H}$ be a connected hypergraph on $n+1$ vertices where $v$ has degree $d\ge 1$ and let $\mathcal{T}$ be a hypertree obtained by Definition~\ref{def:weitz}. Using the same notation as in this definition we additionally define the sets $E_i=\{e_1,\dots,e_i\}\subseteq E(v)=\{e_1,\dots,e_d\}=\{e\in E~|~ v\in e\}$ for $i=0,\dots,d$ and $E'_i=\{e'_1,\dots,e'_i\}\subseteq E'(r)=\{e'_1,\dots,e'_d\}=\{e\in E(\mathcal{T})~|~ r\in e\}$.

Then 
\begin{align*}
R_{v}(\mathcal{H};\bm\lambda)=\frac{Z^\textrm{in}_{v}(\mathcal{H};\bm\lambda)}{Z^\textrm{out}_{v}(\mathcal{H};\bm\lambda)}=\frac{\lambda_v Z(\mathcal{H}\ominus v;\bm\lambda)}{Z(\mathcal{H}-E(v)\ominus v;\bm\lambda)}=\lambda_v\prod_{i=1}^d\frac{Z(\mathcal{H}-E_{i-1}\ominus v;\bm\lambda)}{Z(\mathcal{H}-E_i\ominus v;\bm{\lambda})}
\end{align*}
and by Lemma~\ref{lem:tree_recursion} we have
\begin{align*}
R_{r}(\mathcal{T};\bm\lambda')=\lambda_{v}\prod_{i=1}^d\left(1-\prod_{j=1}^{b_i}\frac{R_{v^{(i)}_j}(\mathcal{T}^{(i)}_{j},\bm\lambda')}{1+R_{v^{(i)}_j}(\mathcal{T}^{(i)}_{j};\bm\lambda')}\right).
\end{align*}

Thus it is sufficient to prove that for each $i=1,\dots,d$ we have 
\[
\frac{Z(\mathcal{H}-E_{i-1}\ominus v;\bm\lambda)}{Z(\mathcal{H}-E_i\ominus v;\bm{\lambda})}=\frac{Z(\mathcal{T}-E'_{i-1}\ominus r;\bm\lambda')}{Z(\mathcal{T}-E'_i\ominus r;\bm{\lambda'})}.
\]
Using the same notations as the Definition~\ref{def:weitz} let us define the sets $S_j=\{v_1^{(i)},\dots,v_{j-1}^{i}\}$ for $j=1,\dots,b_i+1$. Let $\mathcal{H}'=\mathcal{H}-E_{i-1}\ominus v$, then
\begin{align*}
\frac{Z(\mathcal{H}-E_{i-1}\ominus v;\bm\lambda)}{Z(\mathcal{H}-E_{i}\ominus v;\bm\lambda)}&=
\frac{Z(\mathcal{H}';\bm\lambda)}{Z(\mathcal{H}'-\{e_i\};\bm\lambda)}=\frac{Z(\mathcal{H}'-\{e_i\};\bm\lambda)-Z_{e_i}^\textrm{in}(\mathcal{H}'-\{e_i\};\bm\lambda)}{Z(\mathcal{H}'-\{e_i\},\bm\lambda)}\\
&=1-\frac{Z^\textrm{in}_{e_i}(\mathcal{H}'-\{e_i\},\bm\lambda)}{Z(\mathcal{H}'-\{e_i\},\bm\lambda)}.
\end{align*}
Now we claim that for any $i=1,\dots,d$ we have
\[
\frac{Z^\textrm{in}_{e_i}(\mathcal{H}'-\{e_i\},\bm\lambda)}{Z(\mathcal{H}'-\{e_i\},\bm\lambda)}=
\prod_{j=1}^{b_i}\frac{Z_{v^{(i)}_{j}}^\textrm{in}(\mathcal{H}^{(i)}_j;\bm\lambda)}{Z(\mathcal{H}^{(i)}_j;\bm\lambda)}
\]
as rational functions.  If $\mathcal{H}'$ does not have a hyperedge $f$ such that $f\subseteq e_i$, then it holds, because in this case for any $S\subseteq e_i$ we have $Z^\textrm{in}_S(\mathcal{H}'-\{e_i\};\bm\lambda)=Z(\mathcal{H}'-\{e_i\}\ominus S)$ and thus
\[
\frac{Z^\textrm{in}_{e_i}(\mathcal{H}'-\{e_i\},\bm\lambda)}{Z(\mathcal{H}'-\{e_i\},\bm\lambda)}=
 \prod_{j=1}^{b_i}\frac{Z_{S_{j}}^\textrm{in}(\mathcal{H}'-\{e_i\};\bm\lambda)}{Z_{S_{j-1}}^\textrm{in}(\mathcal{H}'-\{e_i\};\bm\lambda)}=\prod_{j=1}^{b_i}\frac{Z_{v^{(i)}_{j}}^\textrm{in}(\mathcal{H}^{(i)}_j;\bm\lambda)}{Z(\mathcal{H}^{(i)}_j;\bm\lambda)}
\]
as rational functions.  Otherwise, if there is a hyperedge contained in $e_i$, then the left-hand side is 0. But in this case the right-hand side is 0 as well. To see this let us choose $j$ to be the smallest integer such that $S_j$ contains a hyperedge of $\mathcal{H}'-\{e_i\}$. Then by construction $\mathcal{H}_j^{(i)}$ contains the hyperedge $\{v_j^{(i)}\}$, which means that
\[
\frac{Z_{v^{(i)}_{j}}^\textrm{in}(\mathcal{H}^{(i)}_j;\bm\lambda)}{Z(\mathcal{H}^{(i)}_j;\bm\lambda)}=0,
\]
proving the claim.

To finish the proof observe that by definition for any $i=1,\dots,d$ and $j=1,\dots,b_i$ we have
\[
    \frac{Z^{\textrm{in}}_{v_j^{(i)}}(\mathcal{H}_j^{(i)};\bm\lambda)}{Z^{\textrm{out}}_{v_j^{(i)}}(\mathcal{H}_j^{(i)};\bm\lambda)}=\frac{Z^{\textrm{in}}_{r_j^{(i)}}(\mathcal{T}_j^{(i)};\bm\lambda')}{Z^{\textrm{out}}_{r_j^{(i)}}(\mathcal{T}_j^{(i)};\bm\lambda')},
\]
thus
\[
    \frac{Z_{v^{(i)}_{j}}^\textrm{in}(\mathcal{H}^{(i)}_j;\bm\lambda)}{Z(\mathcal{H}^{(i)}_j;\bm\lambda)}=\frac{Z_{v^{(i)}_{j}}^\textrm{in}(\mathcal{H}^{(i)}_j;\bm\lambda)}{Z^\textrm{in}(\mathcal{H}^{(i)}_j;\bm\lambda)+Z^\textrm{out}(\mathcal{H}^{(i)}_j;\bm\lambda)}=\frac{R_{v^{(i)}_j}(\mathcal{T}^{(i)}_{j},\bm\lambda')}{1+R_{v^{(i)}_j}(\mathcal{T}^{(i)}_{j};\bm\lambda')}.
\]

\end{proof}

\begin{remark}\label{rem: weitz ratio}
    During the construction of the Weitz-hypertree it could occur that there are hyperedges of size 1 in $\mathcal{T}$. We claim that the theorem remains valid for some linear hypertree $\mathcal{T}'$ where each hyperedge has size at least 2 if $\{v\}\notin E(\mathcal{H})$.  To see this let $(\mathcal{T},r)$ be a Weitz-hypertree of $\mathcal{H}$ with starting vertex $v\in V(\mathcal{H})$. Let $V_1=\{v\in V(\mathcal{T})~|~\{v\}\in E(\mathcal{T})\}$ and let $E_1=\{e\in E(\mathcal{T})~|~e\cap V_1\neq \emptyset\}$. By definition any independent set of $\mathcal{T}$ has to avoid $V_1$, thus $Z(\mathcal{T};\bm\lambda)=Z(\mathcal{T}-E_1-V_1;\bm\lambda)$. The hypergraph $\mathcal{T}-E_1-V_1$ is disjoint union of linear hypertrees $\mathcal{T}_i$ of hyperedge-size at least 2 for $i=1,\dots,k$. Assume that $\mathcal{T}_1$ is the connected component of $r$ in $\mathcal{T}-E_1-V_1$. Now as rational functions we have that
    \[
        R_v(\mathcal{H};\bm\lambda)=R_r(\mathcal{T};\bm\lambda)=R_r(\mathcal{T}-E_1-V_1;\bm\lambda)=R_r(\mathcal{T}_1;\bm\lambda).
    \]
\end{remark}

In the next theorem we will reveal a useful algebraic property of the Weitz-hypertree. It is an extension of an analogous divisibility relation for graphs by the first author of the present paper \cite[Proposition 2.7]{bencs18}.

\begin{theorem}\label{thm:divisibility}
For any connected hypergraph $\mathcal{H}$ and $v\in V(\mathcal{H})$ let  $(\mathcal{T},r)$ be the rooted hypertree with labeling $\pi$ be given by Definition~\ref{def:weitz}.
Then
\[
    Z(\mathcal{H};\bm\lambda) ~|~ Z(\mathcal{T;\bm\lambda'}),
\]
where $\lambda'_u=\lambda_{\pi(u)}$ for $u\in V(\mathcal{T})$. In particular, if $Z(\mathcal{H};\bm\lambda)=0$ for some $\bm\lambda\in\mathbb{C}^{V(\mathcal{H})}$, then $Z(\mathcal{T};\bm\lambda')=0$.

Moreover, there exists subhypergraphs $\{\mathcal{H}_j\}_{j=1}^l$ of $\mathcal{H}$ with vertex set contained in $V(\mathcal{H})\setminus\{v\}$ and integers $k_j\ge 1$ for $j=1,\dots,l$, such that 
\[ 
    Z(\mathcal{{T};\bm\lambda'})=Z(\mathcal{H};\bm\lambda)\prod_{j=1}^lZ(\mathcal{H}_j;\bm\lambda)^{k_j}
\]
\end{theorem}
\begin{proof}
Using the notations of Definition~\ref{def:weitz} we define an ordering on the neighbors of $v$. We say that $u<w$ if and only if either
\[
 \argmin_{i=1,\dots,d}\{u\in e_i\}>\argmin_{i=1,\dots,d}\{w\in e_i\} 
\]
or
\[
 s=\argmin_{i=1,\dots,d}\{u\in e_i\}=\argmin_{i=1,\dots,d}\{w\in e_i\} \qquad\textrm{ and } \qquad\argmin_{j=1,\dots,b_s}\{u=v^{(s)}_j\}<\argmin_{j=1,\dots,b_s}\{w=v_{j}^{(s)}\}.
\]

We know that $Z_v^\textrm{out}(H)$ is the independence polynomial of the hypergraph $\mathcal{H}'=\mathcal{H} - E(v)\ominus v$. Let $\mathcal{H}'_1,\dots,\mathcal{H}'_l$ be the connected components of $\mathcal{H}'$. For each connected component $\mathcal{H}_k$ let $u_k=v_{j_k}^{(i_k)}$ be the smallest neighbor of $v$ appearing in $\mathcal{H}'_k$. Then we see that 
$Z(\mathcal{H}'_k;\bm\lambda)|Z(\mathcal{T}_{j_k}^{(i_k)};\bm\lambda')$ for each $k=1,\dots,l$ and therefore
\[
    Z_v^\textrm{out}(\mathcal{H};\bm\lambda)=Z(\mathcal{H}';\bm\lambda) = \prod_{k=1}^{l} Z(\mathcal{H}'_k;\bm\lambda) ~\Big|~\prod_{k=1}^{l} Z(\mathcal{T}_{j_k}^{(i_k)};\bm\lambda') ~\Big|~ Z_r^\textrm{out}(\mathcal{T};\bm\lambda').
\]
On the other hand we know that as rational functions 
\[
    Z(\mathcal{T};\bm\lambda')=\frac{Z_r^\textrm{out}(\mathcal{T};\bm\lambda')}{Z_v^\textrm{out}(\mathcal{H};\bm\lambda)} Z(\mathcal{H};\bm\lambda)
\]
and thus $Z(\mathcal{H};\bm\lambda)|Z(\mathcal{T};\bm\lambda')$.

The second part follows by induction from the previous equation.
\end{proof}

\begin{remark}\label{rem: weitz divisibility}
    It is not hard to see that if $\mathcal{H}$ is a connected hypergraph and there is no hyperedge $e\in E$, that is contained in an other hyperedge of $\mathcal{H}$, then the rooted hypertree $(\mathcal{T}',r)$ obtained in Remark~\ref{rem: weitz ratio} also satisfies the divisibility relation of the previous theorem, i.e.
    \[
        Z(\mathcal{H};\bm\lambda) ~\big|~ Z(\mathcal{T}';\bm\lambda').
    \]
    In particular, $\mathcal{T}'$ is a linear hypertree of hyperedge-size at least 2.
\end{remark}



\section{Ratios and conditions for zero-freeness}\label{sec: zerofreeness}

In the previous section we saw that a zero of a hypergraph yields a zero for a corresponding linear hypertree. In this section we will show that in fact a zero of a hypergraph implies that there is a rooted linear hypertree with ratio equal to $-1$. It follows from Lemma~\ref{lem:tree_recursion} that the ratios of linear hypertrees can be calculated recursively. Roughly, to show zero-freeness for hypergraphs it is therefore sufficient to show that the ratios of an appropriate class of rooted linear hypertrees stay trapped in a region of the complex plane under application of the appropriate set of maps. We will make this more precise for the multivariate case in Section~\ref{sec: multivariate zero-freeness}, where we will also prove Theorem~\ref{thm: Shearer all hypergraphs}, and for the univariate case in Section~\ref{sec: univariate zero-freeness}.

The key relation between zeros of hypergraphs and $-1$ values of ratios of linear hypertrees is given by the following lemma. This lemma is a generalization of \cite[Lemma 2.1]{BBGPR}, which contains a similar statement for graphs.

\begin{lemma}\label{lem:ratio_vs_zero}
Let $A\subseteq \mathbb{C}\setminus\{-1\}$ be fixed. Then the following are equivalent:
\begin{enumerate}
\item There exists a hypergraph $\mathcal{H}$ of degree at most $\Delta$ and hyperedge-size at most $b+1$ and $\bm\lambda\in A^{V(\mathcal{H})}$ such that $Z(\mathcal{H};\bm\lambda)=0$.
\item There exists a linear hypertree $\mathcal{T}$ of degree at most $\Delta$ and hyperedge-size at most $b+1$ and $\bm\lambda\in A^{V(\mathcal{T})}$ such that $Z(\mathcal{T};\bm\lambda)=0$.
\item There exists a linear hypertree $\mathcal{T}$ of degree at most $\Delta$ and hyperedge-size at most $b+1$, $\bm\lambda\in A^{V(\mathcal{T})}$ and a vertex $r\in V(\mathcal{T})$ of degree 1 such that $R_r(\mathcal{T};\bm\lambda)=-1$.
\end{enumerate}
Moreover, if $\mathcal{H}$ is a $(b+1)-$uniform linear hypertree, then $\mathcal{T}$ can be chosen to be $(b+1)$-uniform as well.
\end{lemma}
\begin{proof}
$(1\Rightarrow 2)$ Assume $\mathcal{H}$ has the stated properties and let $v$ be any vertex of $\mathcal{H}$. By Theorem~\ref{thm:divisibility} we know that the rooted hypertree $\mathcal{T}$ given by Definition~\ref{def:weitz} satisfies that 
\[
    Z(\mathcal{T};\bm\lambda')=0,
\]
where $\lambda'_u=\lambda_{\pi(u)}\in A$. By construction this is a linear hypertree that has degree at most $\Delta$ and hyperedge-size at most $b+1$. 

If $\mathcal{H}$ is a linear uniform hypertree, this implication is trivial.

$(2\Rightarrow 3)$
Now let $\mathcal{T}'$ be a linear hypertree on a minimal number of vertices such that it has degree at most $\Delta$ and hyperedge-size at most $b+1$ and there exists $\bm\lambda\in A^{V(\mathcal{T}')}$ such that $Z(\mathcal{T}',\bm\lambda)=0$. Let $r$ be a leaf vertex of $\mathcal{T}'$ and let $e$ be the unique incident hyperedge at $r$. Since
\[
0=Z(\mathcal{T}';\bm\lambda)=Z_r^\textrm{in}(\mathcal{T}';\bm\lambda)+Z_r^\textrm{out}(\mathcal{T}';\bm\lambda)
\]
and $Z_r^\textrm{out}(\mathcal{T}';\bm\lambda)=Z(\mathcal{T}'-\{e\}- \{r\};\bm\lambda)$ is not zero by the choice of $\mathcal{T}'$, therefore
\[
    R_r(\mathcal{T}';\bm\lambda)=-1.
\]

Observe that if $\mathcal{T}'$ would be a uniform linear hypertree, then $\mathcal{T}''=\mathcal{T}'-\{e\}-r$ would consist of isolated vertices and a uniform linear hypertree, thus $Z(\mathcal{T}'';\bm\lambda)$ is again not zero similarly implying $R_r(\mathcal{T}';\bm\lambda)=-1$.

$(3\Rightarrow 1)$ Let $\mathcal{T}$ be a linear hypertree that satisfies the properties in $(3)$. We claim that $Z(\mathcal{T};\bm\lambda)=0$. If $Z^\textrm{out}_r(\mathcal{T};\bm\lambda)=0$, then $Z^\textrm{in}_r(\mathcal{T};\bm\lambda)$ has to be 0 as well, thus
\[
Z(\mathcal{T};\bm\lambda)=Z^\textrm{in}_r(\mathcal{T};\bm\lambda)+Z^\textrm{out}_r(\mathcal{T};\bm\lambda)=0.
\]
If $Z^\textrm{out}(\mathcal{T};\bm\lambda)\neq 0$, then
\[
Z(\mathcal{T};\bm\lambda)=Z^\textrm{out}_r(\mathcal{T};\bm\lambda)\left(1+\frac{Z^\textrm{in}_r(\mathcal{T};\bm\lambda)}{Z^\textrm{out}_r(\mathcal{T};\bm\lambda)}\right)=0.
\]
\end{proof}

\subsection{Notation and the Grace--Walsh--Szeg\H o Theorem}
\label{sec: Notation}
Fix $d \in \mathbb{Z}_{\geq 2}$ for the remainder of this section. Let $b_1, \dots, b_d \in \mathbb{Z}_{\geq 1}$ and for $i=1, \dots, d$ let $v^i \in (\mathbb{C} \setminus \{-1\})^{b_i}$. We define the map $F_{\lambda}$ by 
\[
    \label{eq: region tilde A}
    F_{\lambda}(v^1, \dots, v^d) = \lambda\cdot\prod_{i=1}^d \left[1-\prod_{j=1}^{b_i} \frac{v^i_j}{1+v^i_j}\right]. 
\]
We say that a closed region $A \subseteq \mathbb{C}\setminus \{-1\}$ is strictly forward invariant for $F_{\lambda}$ if there exists a closed subset $\tilde{A} \subset \Int(A)$ such that for any $b_1, \dots, b_d$ and any $v^i \in A^{b_i}$ we have that 
\[
    F_{\lambda}(v^1, \dots, v^d) \in \tilde{A}.
\]

We denote by $f_{\lambda,b}$ the following univariate specialization of $F_\lambda$
\[
    f_{\lambda,b}(z)=F_\lambda(v(z),\dots,v(z)),
\]
where $v(z)=z\cdot (1,\dots,1)\in \mathbb{C}^{b}$. We will usually drop the subscript $b$ and, unless some other $b$ is specified, $f_\lambda$ will denote $f_{\lambda,1}$. These maps are rational maps, which we will consider as a holomorphic map from the Riemann sphere $\Chat = \mathbb{C} \cup \{\infty\}$ to itself. Since we will be using it often, we also define the M\"obius transformation 
\[
    \mu(z) = \frac{z}{1+z}.
\]
We say that a closed region $A \subseteq \Chat$ is strictly forward invariant for a rational map $g$ if $g(A) \subset \Int (A)$. 

A \emph{generalized circle} is a circle in $\Chat$, i.e. either a circle or a straight line in $\mathbb{C}$. A \emph{circular region} is a subset of either $\Chat$ or $\mathbb{C}$ bounded by a generalized circle that is either open or closed. A circular region in $\mathbb{C}$ is thus either a disk, the complement of a disk or a half-plane. We recall that M\"obius transformation are comformal maps that send generalized circles to generalized circles and thus also circular regions to circular regions. We will denote the closed disk centered around $0$ of radius $r$ by $\B{r}$ and the open unit disk by $\mathbb{D}$. The following theorem by Grace, Walsh and Szeg\H o \cite{grace1902zeros,walsh1922location,szego1922bemerkungen} will be used multiple times throughout the remainder of this paper. Its formulation was taken from \cite{converseGWS}.
\begin{theorem}[Grace--Walsh--Szeg\H o]
\label{thm: GWZ}
Let $f \in \mathbb{C}[z_1, \dots, z_n]$ be a multiaffine polynomial that is invariant under all permutations of the variables. Let $A \subseteq \mathbb{C}$ be a convex circular region. For any $(\zeta_1, \dots, \zeta_n) \in A^n$ there is a $\zeta \in A$ such that 
\[
    f(\zeta_1, \dots, \zeta_n) = f(\zeta, \dots, \zeta).
\]
    
\end{theorem}

\subsection{Zero-free disks for the multivariate independence polynomial}
\label{sec: multivariate zero-freeness}


Let us recall the notation from the introduction
\[
    \lambda_s(\Delta)=\frac{(\Delta-1)^{\Delta-1}}{(\Delta-2)^\Delta}.
\]
Now we will establish one of our main theorems.

\theoremShearer*
\begin{proof}
 By Lemma~\ref{lem:ratio_vs_zero} it is sufficient to prove that  there is no rooted linear hypertree  $(\mathcal{T},r)$ and $\bm\lambda\in \B{\lambda_s(\Delta)}^{V(\mathcal{T})}$, such that
\[
    R_r(\mathcal{T};\bm\lambda)=-1,
\]
where $\mathcal{T}$ has of degree at most $\Delta$ and hyperedge-size at most $b+1$, and the degree of $r$ is 1. By Lemma~\ref{lemma: at least 2} and Remark~\ref{rem: weitz divisibility} it is sufficient to further restrict to linear hypertrees with hyperedge-size at least 2. Thus for the rest of the proof we will consider hypertrees to have hyperedge-size at least 2.

We claim that for any rooted linear hypertree $(\mathcal{T},r)$ of down degree at most $\Delta-1$ we have that $R_r(\mathcal{T},\bm\lambda)$ is well-defined and is an element of $\B{1/\Delta}$. We will prove this by induction on the number of vertices of the hypertree. If $|V(\mathcal{T})|=1$ then $R_r(\mathcal{T},\lambda_r)=\lambda_r\in \B{\lambda_s(\Delta)}\subseteq \B{1/\Delta}$.
Now let us assume that we know that any ratio of hypertrees on $n\ge 1$ vertices with fugacities from $\B{\lambda_s(\Delta)}$ is well-defined and an element of $\B{1/\Delta}$.
Let $(\mathcal{T},r)$ be a hypertree on $n+1$ vertices with down-degree at most $d$ and let $\bm\lambda\in \B{\lambda_s(\Delta)}^{V(\mathcal{T})}$. Then by Lemma~\ref{lem:tree_recursion} we know that
\[
    R_r(\mathcal{T};\bm\lambda)=\lambda_r\prod_{e: r\in e} \left(1-\prod_{v \in e\setminus \{r\}}\frac{R_v(\mathcal{T}_v;\bm\lambda)}{1+R_v(\mathcal{T}_v;\bm\lambda)}\right),
\]
where $\mathcal{T}_v$ is the connected component of $v$ in $\mathcal{T}-E(r)$. By induction, we know that $R_v(\mathcal{T}_v;\bm\lambda)\neq -1$, thus $R_r(\mathcal{T};\bm\lambda)$ is well-defined.
Also we know by induction that for any neighbor $v$ of $r$
\[
    \left|R_v(\mathcal{T}_v;\bm\lambda)\right|\le \frac{1}{\Delta},
\]
thus
\[
    \left|\frac{R_v(\mathcal{T}_v;\bm\lambda)}{1+R_v(\mathcal{T}_v;\bm\lambda)}\right|\le  \frac{1}{\Delta-1},
\]
therefore for any edge $e$ that is incident to $v$ we have
\[
    \left|\prod_{v\in e\setminus\{r\}}\frac{R_v(\mathcal{T}_v;\bm\lambda)}{1+R_v(\mathcal{T}_v;\bm\lambda)}\right|\le  \frac{1}{\Delta-1},
\]
since $1/(\Delta-1)\le 1$. Thus,
\[
    \left|R_r(\mathcal{T};\bm\lambda)\right|=|\lambda_r|\prod_{e: r\in e} \left|1-\prod_{v \in e\setminus \{r\}}\frac{R_v(\mathcal{T}_v;\bm\lambda)}{1+R_v(\mathcal{T}_v;\bm\lambda)}\right|\le \lambda_s(\Delta)\left(1+\frac{1}{\Delta-1}\right)^{\Delta-1}\le \frac{1}{\Delta-1}
\]

\end{proof}

Let us remark that in the previous argument we actually proved that for any $\lambda\in \B{\lambda_s(\Delta)}$ the disk $\B{1/\Delta}$ is $F_\lambda$ forward invariant (where $\Delta = d+1$). In the next lemma we have a similar statement, but for $(b+1)$ uniform hypertrees.

\begin{lemma}
\label{lem: zero-free disk k-uniform hypertrees}
Let $b\ge 1$. Suppose there exists an $0<r<1$ such that  $f_{\lambda,b}(\B{R})\subseteq \B{R}$. Then for any $(b+1)$-uniform hypertree $\mathcal{T}$ of maximum degree $\Delta=d+1$ and for any $\bm\lambda \in \B{|\lambda|}^{V(\mathcal{T})}$ we have $Z(\mathcal{T};\bm\lambda)\neq 0$.
\end{lemma}
\begin{proof}
As in the previous proof, it is sufficient to prove that $R_r(\mathcal{T};\bm\lambda)$ is well-defined and not $-1$ for some $r\in V(\mathcal{T})$ leaf. To be more precise we claim that all the ratios of $(b+1)$-uniform hypertrees of maximum degree $\Delta$ and for any leaf $r$ the corresponding ratio is contained in $\B{R}$. To prove this we show that the one vertex hypertree has ratio in $\B{R}$ and $\underbrace{\B{R}\times \dots\times \B{R}}_{b}$ is mapped by $F_{\lambda'}$ into $\B{R}$ for any $\lambda'\in \B{|\lambda|}$. This is sufficient, since now by induction on the number of vertices it follows that any ratio is well-defined and contained in $\B{R}$.

If $\mathcal{T}$ is an isolated vertex, then $R_r(\mathcal{T};\bm\lambda)=\lambda_v\in \B{|\lambda|}$.
To prove the second part  let us consider the product
\[
    \prod_{j=1}^{b}\frac{v_j^i}{1+v_j^i},
\]
where $v_j^i\in \B{R}$. Since $\mu(z)=\frac{z}{1+z}$ is a M\"obius transformation, therefore $\mu(v_j^i)$ is an element of the disk $\mu(\B{R})$ for $j=1,\dots,b$ and $i=1,\dots,d$. As the map $g_1(z_1,\dots,z_b)=z_1\dots z_b$ is a symmetric polynomial, therefore by Theorem~\ref{thm: GWZ} we know that there exists $w_i\in\mu(\B{R})$ such that
\[
    \prod_{j=1}^{b}\frac{v_j^i}{1+v_j^i}=(w_i)^b,
\]
i.e.
\[
    F_{\lambda'}(v^1,\dots,v^d)=\lambda'\prod_{i=1}^d(1-w_i^b).
\]
Now let $w\in \mu(\B{R})$ be chosen such that it maximizes $|1-z^b|$ over $\mu(\B{R})$.
Thus, if $v=\mu^{-1}(w)\in \B{R}$, then
\[
    \left|F_{\lambda'}(v^1,\dots,v^d)\right|\le |\lambda'|\left|\prod_{i=1}^d(1-w^b)\right|\le |\lambda|\left|\prod_{i=1}^d\left(1-\left(\frac{v}{1+v}\right)^b\right)\right|=|f_{\lambda,b}(v)|\le R.
\]
\end{proof}
\subsection{A zero-freeness condition for the univariate independence polynomial}
\label{sec: univariate zero-freeness}

In this section we describe a sufficient condition for showing that $\lambda$ cannot be a zero of a bounded degree hypergraph. We recall that $d \in \mathbb{Z}_{\geq 2}$ is assumed to be fixed and that the maps $F$ and $f$ are defined in Section~\ref{sec: Notation}.

\begin{lemma}
Suppose there is a closed strictly forward invariant region $A \subseteq \{z \in \mathbb{C}: \Re(z) \geq -\frac{1}{2}\}$ for $F_{\lambda_0}$ for some $\lambda_0 \in \mathbb{C}\setminus \{0\}$. Then there is a neighborhood $U$ of $\lambda_0$ such that $A$ is strictly forward invariant for $F_\lambda$ for all $\lambda \in U$.
\end{lemma}

\begin{proof}
Observe that $\mu(\{z \in \mathbb{C}: \Re(z) \geq -\frac{1}{2}\})$ is equal to the closed unit disk $\overline{\mathbb{D}}$. Therefore, if we let $g(w) = \lambda_0 \cdot (1-w)^d$, we see that 
\[
    F_{\lambda_0}(v^1, \dots, v^d) \in g(\overline{\mathbb{D}})
\]
for every choice $v^i \in A^{b_i}$. This is a bounded region, which implies that the region $\tilde{A}$, described in equation (\ref{eq: region tilde A}), can be chosen to be bounded. It follows that there is an open neighborhood $V$ of $1$ such that $t \cdot \tilde{A}$ is strictly contained in $A$ for every $t \in V$. Thus, if we let $U = \lambda_0 \cdot V$, we see that for every $\lambda \in U$
\[
    F_{\lambda_0}(v^1, \dots, v^d) = \frac{\lambda}{\lambda_0}F_{\lambda}(v^1, \dots, v^d) \in \frac{\lambda}{\lambda_0} \tilde{A},
\]
which is a strict subset of $A$ because $\frac{\lambda}{\lambda_0} \in V$.

\end{proof}

We say that a set $A \subseteq \mathbb{C} \setminus \mathbb{R}_{<0}$ is \emph{log-convex} if $\log(A \setminus \{0\})$ is convex, where $\log$ is taken to be the principle branch of the logarithm defined on $\mathbb{C} \setminus \mathbb{R}_{\leq 0}$.

A set $B \subseteq \mathbb{C}$ is called a \emph{multiplicative semigroup} if for every $w_1,w_2 \in B$ also $w_1 \cdot w_2 \in B$.

\begin{lemma}
\label{lem: properties of A}
Let $\lambda_0 \in \mathbb{C}$ and suppose a closed set $A\subseteq \mathbb{C} \setminus \mathbb{R}_{\leq -1}$ satisfies the following properties:
\begin{itemize}
    \item
    $1 + A$ is $\log$-convex;
    \item
    $\mu(A)$ is a multiplicative semigroup;
    \item
    $A$ is strictly forward invariant for the univariate map $f_{\lambda_0}$
\end{itemize}
Then $A$ is strictly forward invariant for $F_{\lambda_0}$.
\end{lemma}

\begin{proof}
Let $v^i\in A^{b_i}$ for $i=1,\dots,d$. By the second property we know that there exists $w_i\in A$ for $i=1,\dots,d$ such that
\[
    \prod_{j=1}^{b_i}\frac{v^{i}_j}{1+v^{i}_j}=\frac{w_i}{1+w_i}.
\]
By the first property we know that there exists a $\tilde{w}\in A$ such that
\[
    (1+\tilde{w})^d=\prod_{i=1}^d (1+w_i).
\]
Therefore
\[
    F_{\lambda_0}(v^1,\dots,v^d)=\lambda_0\prod_{i=1}^d \left[1-\prod_{j=1}^{b_i} \frac{v^i_j}{1+v^i_j}\right]=\lambda_0\prod_{i=1}^d\frac{1}{1+w_i}=\frac{\lambda_0}{(1+\tilde{w})^d}=f_{\lambda_0}(\tilde w).
\]
And thus any $F_{\lambda_0}(v^1,\dots,v^d)$ lies in $f_{\lambda_0}(A)$, which is a strict subset of $A$ by the first property.
\end{proof}

\begin{corollary}
\label{cor: properties imply zerofree}
Suppose $\lambda_0 \in \mathbb{C}$ and $A\subseteq \{z \in \mathbb{C}: \Re(z) \geq -\frac{1}{2}\}$ satisfy the properties of Lemma~\ref{lem: properties of A} and $0 \in A$. Then there is an open neighborhood $U$ of $\lambda_0$ such that for any hypergraph $\mathcal{H}$ of maximum degree $d+1$ and $\lambda \in U$ we have $Z(\mathcal{H};\lambda) \neq 0$.
\end{corollary}

\begin{proof}
By Lemma~\ref{cor: properties imply zerofree} and Lemma~\ref{lem: properties of A} we know that there is an open neighborhood $U$ of $\lambda_0$ such that $A$ is strictly forward invariant for $F_\lambda$ for any $\lambda\in U$. 
Now let us fix a $\lambda\in U$. 
Let $\mathcal{T}$ be an arbitrary hypertree of maximum degree at most $d+1$ and let $r$ be a leaf. Then by the forward invariance of $A$ and the fact that $-1\notin A$ we see by Lemma~\ref{lem:tree_recursion} that $R_r(\mathcal{T},\lambda)\neq -1$. Thus by Lemma~\ref{lem:ratio_vs_zero} we see that there is no hypergraph of degree at most $d+1$ such that $Z(\mathcal{H},\lambda)=0$.
\end{proof}

\section{The maximal zero-free region for hypergraphs around the positive real axis}\label{sec: Sokal}
This section is dedicated to proving Theorem~\ref{thm: hypergraph Sokal region}. Again we will always fix $d \in \mathbb{Z}_{\geq 2}$ and let $f_{\lambda}(z) = \lambda/(1+z)^d$. We let
\[
    \lambda_c := \lambda_c(d+1) = \frac{d^d}{(d-1)^{d+1}}.
\]
We can relate $\lambda_c$ to the dynamical behaviour of $f_\lambda$. The following lemma can also be found in \cite{PetersRegtsSokal}.
\begin{lemma}
    \label{lem: lambdacritical}
    For any $\lambda > 0$ the map $f_{\lambda}$ has a unique positive fixed point $p_\lambda$. If $0 < \lambda < \lambda_c$ then this fixed point is attracting, i.e. $|f_{\lambda}'(p_\lambda)| < 1$.
\end{lemma}
\begin{proof}
Suppose $f_{\lambda}$ has a fixed point $p \in \mathbb{R}_{\geq 0}$. Then it follows that $\lambda = \lambda_p = p(1+p)^d$. The map $p \mapsto \lambda_p$ is increasing on $\mathbb{R}_{\geq 0}$ and therefore invertible with inverse $p_\lambda$. 

We calculate that $f_{\lambda_p}'(p_\lambda) = -dp_\lambda/(1+p_\lambda)$. It therefore follows that the fixed point $p_\lambda$ of $f_{\lambda_p}$ is attracting if $0 < p_\lambda < 1/(d-1)$, i.e. if $0 < \lambda_p < \lambda_c$. 
\end{proof}

To prove Theorem~\ref{thm: hypergraph Sokal region} we will 
define a family of regions $A(x,x_0,\epsilon)$ depending on three parameters and we will show that for $0 < \lambda < \lambda_c$ the parameters can be chosen in such a way that the region satisfies the conditions of Lemma~\ref{lem: properties of A}.

\subsection{Defining the relevant regions}

For any $x \in \mathbb{R}$ and $\epsilon \in (0,\pi)$ we define $C(x,\epsilon)$ as the closed cone with vertex $x$ that is symmetric around the real axis and has internal angle $2\epsilon$, i.e.
\[
    C(x,\epsilon) = \{z \in \mathbb{C}: |\textrm{Arg}(z-x)| \leq \epsilon\} \cup \{x\}.
\]
Let $x,x_0 \in \mathbb{R}$ with $-1 < x < x_0$ and let $\epsilon \in (0,\pi/2)$. Let $\tilde{\epsilon}$ be such that the intersections of the boundaries of $C(x,\epsilon)$ and $C(-1,\tilde{\epsilon})$ have real part $x_0$, i.e $\tilde{\epsilon} = \tan^{-1}(\frac{x_0-x}{x_0+1}\tan(\epsilon))$.
We define 
\[
    A(x,x_0,\epsilon) = C(x,\epsilon) \cap C(-1,\tilde{\epsilon}).
\]

For $y < 1$ let $a_\epsilon$ be the circular arc between $y$ and $1$ with internal angle $\epsilon$. And let $D(y,\epsilon)$ be the closed region in between $a_\epsilon$ and $\overline{a_\epsilon}$. For $y_0 \in (y,1)$ let $\tilde{\eta}$ be such that the real part of the intersections of the boundary between the cone $-C(-1,\tilde{\eta})$ and $D(y,\epsilon)$ is $y_0$.
We define 
\[
B(y,y_0,\epsilon) = D(y,\epsilon) \cap (-C(-1,\tilde{\eta})).
\]
See figure \ref{fig: Regions} for an example.

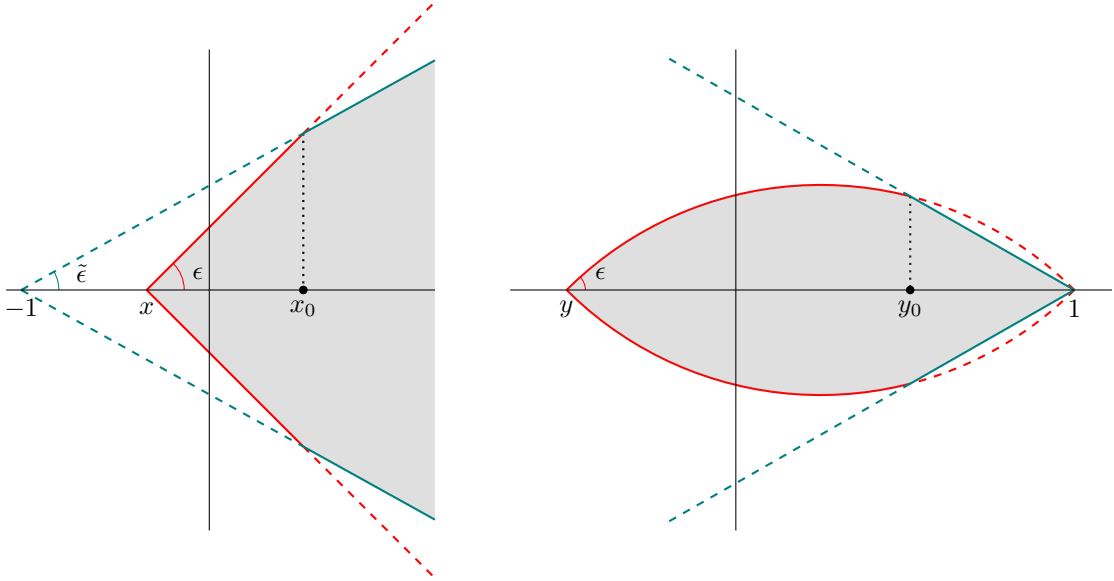
\begin{figure}[ht]
\centering
\label{fig: Regions}
\begin{tikzpicture}

\fill[fill=gray, fill opacity = 0.25] (-0.833333,0.) -- (1.25,2.08333) -- (3.,3.05556) -- (3.,-3.05556) -- (1.25,-2.08333) -- (-0.833333,0.);

\draw[thick,black,dotted] (1.25,2.08333) -- (1.25,0);

\draw[thick,red] (-0.833333,0.) -- (1.25,2.08333);
\draw[thick,red] (-0.833333,0.) -- (1.25,-2.08333);
\draw[thick,red,dashed] (1.25,2.08333) -- (3.,3.83333);
\draw[thick,red,dashed] (1.25,-2.08333) -- (3.,-3.83333);

\draw[thick,teal,dashed] (-2.5,0.) -- (1.25,2.08333);
\draw[thick,teal,dashed] (-2.5,0.) -- (1.25,-2.08333);
\draw[thick,teal] (1.25,2.08333) -- (3.,3.05556);
\draw[thick,teal] (1.25,-2.08333) -- (3.,-3.05556);

\node[black] at (-0.15,0.22){$\epsilon$};
\draw[red,-] (-0.833333,0.)+(0:0.5) arc (0:45:0.5);
\node[black] at (-1.7,0.22){$\tilde\epsilon$};
\draw[teal,-] (-2.5,0.)+(0:0.5) arc (0:29.0546:0.5);
\node[black] at (-2.5,-0.25){$-1$};
\node[black] at (-0.833333,-0.25){$x$};
\draw[fill=black] (1.25,0) circle (0.05);
\node[black] at (1.25,-0.25){$x_0$};

\draw[black] (-2.7,0.) -- (3,0.);
\draw[black] (0.,-3.2) -- (0.,3.2);

\draw [red,thick,domain=75.5225:135.] plot ({8.125+4.77297*cos(\x)}, {-3.375+4.77297*sin(\x)});
\draw [red,dashed,thick,domain=45.:75.5225] plot ({8.125+4.77297*cos(\x)}, {-3.375+4.77297*sin(\x)});
\draw [red,thick,domain=-75.5225:-135.] plot ({8.125+4.77297*cos(\x)}, {--3.375+4.77297*sin(\x)});
\draw [red,dashed,thick,domain=-45.:-75.5225] plot ({8.125+4.77297*cos(\x)}, {--3.375+4.77297*sin(\x)});

\draw[dotted,thick,black] (9.31824,1.24641) -- (9.31824,0);

\fill[fill=gray, fill opacity = 0.25] (4.75,0.) -- (5.11789,0.33155) -- (5.51816,0.6232) -- (5.94649,0.8718) -- (6.39828,1.07468) -- (6.86866,1.22965) -- (7.35256,1.33505) -- (7.84478,1.38974) -- (8.34002,1.39312) -- (8.83294,1.34518) -- (9.31824,1.24641) -- (11.5,0) -- (9.31824,-1.24641) -- (8.83294,-1.34518) -- (8.34002,-1.39312) -- (7.84478,-1.38974) -- (7.35256,-1.33505) -- (6.86866,-1.22965) -- (6.39828,-1.07468) -- (5.94649,-0.8718) -- (5.51816,-0.6232) -- (5.11789,-0.33155) -- (4.75,0.);

\draw[thick,teal] (11.5,0) -- (9.31824,1.24641);
\draw[thick,dashed,teal] (9.31824,1.24641) -- (6.0456, 3.11603);
\draw[thick,teal] (11.5,0) -- (9.31824,-1.24641);
\draw[thick,dashed,teal] (9.31824,-1.24641) -- (6.0456, -3.11603);

\node[black] at (5.2,0.22){$\epsilon$};
\draw[red,-] (4.75,0.)+(0:0.25) arc (0:45:0.25);
\node[black] at (11.5,-0.25){$1$};
\node[black] at (4.75,-0.25){$y$};
\draw[fill=black] (9.31824,0) circle (0.05);
\node[black] at (9.31824,-0.25){$y_0$};


\draw[black] (4,0.) -- (12,0.);
\draw[black] (7.,-3.2) -- (7.,3.2);
\end{tikzpicture}
\caption{Example of the regions $A(x,x_0,\epsilon)$ and $B(y,y_0,\epsilon)$}
\end{figure}

Recall that $\mu(z) = z/(1+z)$.

\begin{lemma}
\label{lem: Region Transform}
For $-1 < x < x_0$ we have
\[
    \mu(A(x,x_0,\epsilon)) = B(\mu(x),\mu(x_0) + o(1),\epsilon) \setminus \{1\},
\]
where $o(1)$ denotes some function that converges to $0$ as $\epsilon$ converges to $0$.
\end{lemma}

\begin{proof}
The map $\mu$ is a M\"obius transformation that sends $\mathbb{R} \cup \{\infty\}$ to itself with $\mu(-1) = \infty$ and $\mu(\infty) = 1$. A M\"obius transformation is a conformal map that sends generalized circles to generalized circles and therefore, for any $-1 < x$, we have that $\mu(C(x,\epsilon)) = D(\mu(x),\epsilon) \setminus \{1\}$. Moreover, it follows that $\mu(C(-1,\tilde{\epsilon})) = -C(-1,\tilde{\epsilon}) \setminus \{1\}$. If we let $I_\epsilon$ be the point on the boundary of $A(x,x_0,\epsilon)$ in the upper half-plane with real part $x_0$ we can thus conclude that 
\[
    \mu(A(x,x_0,\epsilon)) = B(\mu(x), \Re(\mu(I_\epsilon)),\epsilon) \setminus \{1\}.
\]
As $\epsilon \to 0$ the point $I_\epsilon$ converges to $x_0$ and thus, by continuity, $\mu(I_\epsilon)$ converges to $\mu(x_0)$.
\end{proof}

\subsection{Logarithmic convexity}
In this section we prove that $A(x,x_0,\epsilon)$ satisfies the first condition of Lemma~\ref{lem: properties of A}.

\begin{lemma}
\label{lem: log convexity}
    For any $x,x_0$ with $-1 < x < x_0$ and $\epsilon \in (0,\pi/2)$ the region $1 + A(x,x_0,\epsilon)$ is $\log$-convex.
\end{lemma}

\begin{proof}
Because $A(x,x_0,\epsilon)$ is the intersection of two cones it is sufficient to show that any cone of the form $C(y,\epsilon)$ with $y \geq 0$ and $\epsilon \in (0, \pi/2)$ is $\log$-convex. Furthermore, such a cone is the intersection of the two closed half-planes $y - e^{i\epsilon} \mathbb{H}_{\geq 0}$ and $y + e^{-i\epsilon} \mathbb{H}$, where $\mathbb{H}$ denotes the closed upper half-plane. It is thus sufficient to show that the half-plane $H(y,\epsilon) := y + e^{-i\epsilon} \mathbb{H}$ is log-convex.

If $y = 0$ then 
\[
\log(H(y,\epsilon) \setminus \{0\}) = \log(e^{-i\epsilon} \mathbb{H}_{\geq 0}   \setminus \{0\}) = \{w \in \mathbb{C}: -\epsilon \leq \Im(w) \leq \pi - \epsilon\},
\]
which is convex. 

Now suppose $y > 0$ and take $w_1, w_2 \in \log(H(y,\epsilon))$. By Theorem~\ref{thm: GWZ} there exists a $z \in H(y,\epsilon)$ such that $e^{w_1} e^{w_2} = z^2$. It follows that $(w_1 + w_2)/2 = \log(z) + i k \pi$ for some $k \in \mathbb{Z}$. Note that for any $w \in \log(H(y,\epsilon))$ we have $-\epsilon < \Im(w) < \pi - \epsilon$ and thus both $- \epsilon < \Im((w_1 + w_2)/2) < \pi - \epsilon$ and $-\epsilon < \Im(\log(z)) < \pi - \epsilon$. It follows that $k=0$, i.e. $(w_1 + w_2)/2 = \log(z)$, which shows that $\log(H(y,\epsilon))$ is midpoint convex. Because $\log(H(y,\epsilon))$ is also closed we can conclude that $H(y,\epsilon)$ is convex.
\end{proof}

\subsection{Multiplicative invariance}
In this section we prove that for the right choice of parameters the set $A(x,x_0,\epsilon)$ satisfies the second condition of Lemma~\ref{lem: properties of A}.

\begin{lemma}
    \label{lem: B mult semigroup}
    Let $y \in (0,1)$, and let $r$ be a real-valued function such that $\lim_{\epsilon \to 0} r(\epsilon) = y_0$ with $y_0 \in (y^2,\sqrt{y})$. Then for $\epsilon$ sufficiently small $B(-y,r(\epsilon),\epsilon)$ is a multiplicative semigroup.
\end{lemma}

\begin{proof}
Let $\mathcal{B}_{\epsilon} = B(-y,r(\epsilon),\epsilon)$. We have to show that for $\epsilon$ sufficiently small and $z_1, z_2 \in \mathcal{B}_\epsilon$ the product $z_1 z_2 \in \mathcal{B}_\epsilon$. Because $\mathcal{B}_\epsilon$ is convex and contains $0$ we can assume that $z_1, z_2 \in \partial \mathcal{B}_\epsilon$. The boundary of $\mathcal{B}_\epsilon$ is the union of the closed circular arcs $a_\epsilon$ and $\overline{a}_\epsilon$ and the closed straight line segments $b_\epsilon$ and $\overline{b_\epsilon}$. We claim that we can assume that either $z_1, z_2 \in a_\epsilon$ or $z_1 \in a_\epsilon$ and $z_2 \in \overline{a_\epsilon}$. 

To show this let us denote the unique element of the intersection $a_\epsilon \cap b_\epsilon$ by $I_\epsilon$. Assume that $z_1$ is an element of either $b_\epsilon$ or $\overline{b_\epsilon}$, say $z_1 \in b_\epsilon$. The set $z_2 \cdot b_\epsilon$ is a straight line segment between $z_2$ and $I_\epsilon z_2$ containing $z_1 z_2$. Because $z_2 \in \mathcal{B}_\epsilon$ and $\mathcal{B}_\epsilon$ is convex it follows that $z_1 z_2 \in \mathcal{B}_\epsilon$ if $I_\epsilon z_2 \in \mathcal{B}_\epsilon$. We can thus assume that $z_1$ is either an element of $a_\epsilon$ or of $\overline{a_\epsilon}$. Analogously we can conclude the same for $z_2$. By symmetry of $\mathcal{B}_\epsilon$ under complex conjugation it suffices to only consider the two cases $z_1, z_2 \in a_\epsilon$ or $z_1 \in a_\epsilon$ and $z_2 \in \overline{a_\epsilon}$ as indicated in the claim. 

We will now focus on the case where $z_1$ and $z_2$ are both an element of $a_\epsilon$. The curve $a_\epsilon$ is a circular arc and can thus be written as $\{t+ f_\epsilon(t)i: t \in [-y,r(\epsilon)]\}$ for a real valued function $f_\epsilon$ depending on $\epsilon$. Note that this function naturally extends to the whole interval $[-y,1]$ and that $f_\epsilon(t) \in \mathcal{O}(\epsilon)$ uniformly over all $t \in [-y,1]$. Fix $y_u \in (y_0, \sqrt{y})$ such that $y_u^2 < y_0$ and in what follows let $\epsilon$ be sufficiently small such that $y_u^2 < r(\epsilon) < y_u$. We will show that for $\epsilon$ sufficiently small the product of two elements on the curve $\{t + f_\epsilon(t)i: t \in [-y,y_u]\}$ lies in $\mathcal{B}_\epsilon$. This implies that $z_1 z_2 \in \mathcal{B}_\epsilon$. Consider the product of two such arbitrary elements
\[
    (t_1 + f_\epsilon(t_1)i) \cdot (t_2 + f_\epsilon(t_2)i) = t_1t_2 + (t_1f_\epsilon(t_2) + t_2 f_\epsilon(t_1)) i + \mathcal{O}(\epsilon^2).
\]
We will show that there is a constant $C \in (0,1)$ such that for $\epsilon$ sufficiently small  $t_1t_2 + (t_1f_\epsilon(t_2) + t_2 f_\epsilon(t_1)) i$ is an element of
\[
    S_\epsilon = \{a + b i : -y y_u \leq a \leq y_u^2 \text{ and } |b| \leq C \cdot f_\epsilon(a)\}
\]
for every $t_1,t_2 \in [-y,y_u]$. This is sufficient because $S_\epsilon \subset \mathcal{B}_\epsilon$ and the distance between $\partial S_\epsilon$ and $\partial \mathcal{B}_\epsilon$ is lower bounded by a constant times $\epsilon$.

For $t_1,t_2 \in [-y,y_u]$ we have that $t_1 t_2$ is an element of $[-y y_u, y_u^2] \cup [-y y_u, y^2] = [-y y_u, y_u^2]$. What remains to be shown is that there exists a constant $C \in (0,1)$ such that for small $\epsilon$ we have 
\begin{equation}
    \label{eq: inequality1}
    |t_1f_\epsilon(t_2) + t_2 f_\epsilon(t_1)| \leq C \cdot f_\epsilon(t_1 t_2).
\end{equation}
To prove this inequality we may replace the function $t \mapsto f_\epsilon(t)$ by a function $t \mapsto \gamma_\epsilon f_\epsilon(t)$, where $\gamma_\epsilon$ is any positive valued function of $\epsilon$. Let
$\gamma_\epsilon = \frac{(1-y)(1+3y)}{4f_{\epsilon}((1+y)/2))}$ and define $g_\epsilon(t) = \gamma_\epsilon f_{\epsilon}(t)$. Now the curve $t + i g_{\epsilon}(t)$ for $t \in [-y,1]$ defines part of an ellipse that passes through $-y$, $1$ and $\frac{1}{2}(1+y) + \frac{1}{4}i(1-y)(1+3y)$ for every $\epsilon$. As $\epsilon$ goes to zero the eccentricity of this ellipse converges to $1$. It follows that the elliptical arcs converge uniformly to the arc of a parabola. Because this parabola has to go through the same three points it must be given by the equation $t + i g(t)$, where $g(t) = -(t+y)(t-1)$. We claim that, to obtain equation (\ref{eq: inequality1}), it is now sufficient to prove that
\begin{equation}
    \label{eq: inequality2}
   |t_1g(t_2) + t_2 g(t_1)| < g(t_1 t_2).
\end{equation}
for all $t_1, t_2 \in [-y,y_u]$. Indeed, if this inequality holds then, because $[-y,y_u]$ is compact, there is also some $C \in (0,1)$ such that $|t_1g(t_2) + t_2 g(t_1)| < C \cdot g(t_1 t_2)$. By uniform continuity the same will then hold for $g_\epsilon$ for $\epsilon$ sufficiently small and thus also for $f_\epsilon$.

At this point we recall that we arrived at the inequality in equation (\ref{eq: inequality2}) by assuming that $z_1$ and $z_2$ are both elements of $a_\epsilon$. If we assume the other case, that is, where $z_1 \in a_\epsilon$ and $z_2 \in \overline{a_\epsilon}$, then, instead of considering $(t_1 + f_\epsilon(t_1)i) \cdot (t_2 + f_\epsilon(t_2)i)$, we consider $(t_1 + f_\epsilon(t_1)i) \cdot (t_2 - f_\epsilon(t_2)i)$. In a completely analogous way it then follows that it is sufficient to prove the following inequality
\begin{equation}
    \label{eq: inequality3}
   |-t_1g(t_2) + t_2 g(t_1)| < g(t_1 t_2).
\end{equation}
for all $t_1, t_2 \in [-y,y_u]$. To prove the inequalities in both equation (\ref{eq: inequality2}) and (\ref{eq: inequality3}) we will prove that 
\[
    |t_1|g(t_2) + |t_2| g(t_1) < g(t_1 t_2).
\]
for all $t_1, t_2 \in [-y,y_u]$. We distinguish three cases.
\begin{itemize}
    \item 
    If $t_1$ and $t_2$ are both positive then 
    \[
        g(t_1 t_2) - |t_1|g(t_2) - |t_2| g(t_1) = g(t_1 t_2) - t_1g(t_2) - t_2 g(t_1) = (1-t_1)(1-t_2)(y-t_1t_2).
    \]
    This is strictly positive because $t_1t_2 \leq y_u^2 < y$.
    \item
    If $t_1$ and $t_2$ have opposite signs, say $t_2 \leq 0$, then 
    \[
        g(t_1 t_2) - |t_1|g(t_2) - |t_2| g(t_1) = g(t_1 t_2) - t_1g(t_2) + t_2 g(t_1) = (1-t_1)(1+t_2)(y+t_1t_2),
    \]
    which is again strictly positive because $t_1t_2 \geq -y y_u > -y$.
    \item
    If $t_1$ and $t_2$ are both strictly negative then
    \begin{align*}
        g(t_1 t_2) - |t_1|g(t_2) - |t_2| g(t_1) &= g(t_1 t_2) + t_1g(t_2) + t_2 g(t_1)\\ 
        &= t_1 t_2 (3-t_1-t_2-t_1 t_2) +  y(1+t_1 + t_2 -3 t_1 t_2).
    \end{align*}
    Note that $t_1 t_2 (3-t_1-t_2-t_1 t_2)$ is strictly positive and thus if $(1+t_1 + t_2 -3 t_1 t_2)$ is also positive the whole quantity is strictly positive. Otherwise we have  
    \begin{align*}
        t_1 t_2 (3-t_1-t_2-t_1 t_2) +  y(1+t_1 + t_2 -3 t_1 t_2) &\geq t_1 t_2 (3-t_1-t_2-t_1 t_2) +  (1+t_1 + t_2 -3 t_1 t_2)\\
        &=(1+t_1)(1+t_2)(1-t_1t_2),
    \end{align*}
    which is strictly positive.
\end{itemize}
This concludes the proof.
\end{proof}

\begin{corollary}
    \label{cor: multiplicative semigroup}
    Let $x \in (-\frac{1}{2},0)$ and $x_0$ such that 
    \begin{equation}
        \label{eq: condition semigroup}
        \mu^{-1}(\mu(x)^2) < x_0 < \mu^{-1}(\sqrt{-\mu(x)}).
    \end{equation}
    Then for $\epsilon$ sufficiently small $\mu(A(x,x_0,\epsilon))$ is a multiplicative semigroup.
\end{corollary}

\begin{proof}
Lemma~\ref{lem: Region Transform} states that $
\mu(A(x,x_0,\epsilon)) = B(\mu(x),\mu(x_0) + o(1),\epsilon) \setminus \{1\}.$
Note that $-\mu(x)\in (0,1)$ and thus, by Lemma~\ref{lem: B mult semigroup}, it is sufficient to have 
\[
    (-\mu(x))^2 < \mu(x_0) < \sqrt{-\mu(x)}.
\]
Because $\mu^{-1}$ is orientation preserving this condition is equivalent to the condition of equation~{(\ref{eq: condition semigroup})}.

\end{proof}

\subsection{Forward invariance}
In this section we prove that for the right choice of parameters the set $A(x,x_0,\epsilon)$ satisfies the third condition of Lemma~\ref{lem: properties of A}.

\begin{lemma}
    \label{lemma: cone to cone}
    Let $\lambda > 0$, $x > -1$ and $\delta < \pi/d$, then $f_{\lambda}(C(x,\delta)) \subseteq C(0, d \cdot \delta)$.
\end{lemma}
\begin{proof}
The map $z \mapsto z+1$ sends the cone $C(x,\delta)$ to $C(x+1,\delta)$, which is a subset of $C(0,\delta)\setminus\{0\}$. Subsequently, the map $z \mapsto z^{-d}$ sends $C(0,\delta)\setminus\{0\}$ to $C(0,d\cdot \delta)\setminus\{0\} \subset C(0, d \cdot \delta)$. Finally, because $\lambda$ is a positive real number, multiplication with $\lambda$ sends $C(0, d \cdot \delta)$ to itself. This concludes the proof.
\end{proof}

\begin{lemma}
    \label{lem: Cone forward invariant}
    Let $0 < \lambda < \lambda_c$. For $\epsilon$ sufficiently small the cone $C(-\frac{1}{d+1},\epsilon)$ is strictly forward invariant for $f_{\lambda}$.
\end{lemma}
\begin{proof}
For any $\delta > 0$ define the open strip \[S_{<\delta} = \{z \in \mathbb{C} : |\Im(z)| < \delta\}\] and let $S_{\leq \delta}$ be its closure. We define the map
\[
\phi:\ \mathbb{C}\setminus \mathbb{R}_{\leq -1/(d+1)} \to S_{< \pi},
\ 
z \mapsto \log\left(z+\frac{1}{d+1}\right).
\]
This is a biholomorphism with $\phi^{-1}(w) = e^w-\frac{1}{d+1}$. Note that for any $\delta \in (0,\pi)$ we have that $\phi^{-1}(S_{\leq \delta}) = C(-1/(d+1),\delta) \setminus \{-\frac{1}{d+1}\}$. Now fix $\delta \in (0,\pi/d)$, it follows from Lemma~\ref{lemma: cone to cone} that the map $g_\lambda := \phi\circ f_{\lambda} \circ \phi^{-1}$ is well-defined as a map $S_{\leq \delta} \to S_{\leq d \delta}$. 

We claim that there exist constants $c \in (0,1)$ and $\epsilon_0 \in (0,\delta)$ such that $|g_\lambda'(w)| \leq c$ for all $z \in S_{\leq \epsilon_0}$. If we assume that this claim is true then it follows that for any $w \in S_{\leq \epsilon}$ with $0 < \epsilon \leq \epsilon_0$ we have that 
\[
    |\Im(g_\lambda(w))| = |\Im(g_\lambda(w) - g_\lambda(\Re(w)))| \leq |g_\lambda(w) - g_\lambda(\Re(w))|
    \leq c \cdot |w-\Re(w)| \leq c \cdot \epsilon.
\]
Therefore $g_\lambda(S_{\leq \epsilon}) \subseteq S_{\leq c \epsilon}$ and thus, by conjugation, $f_{\lambda}(C(-1/(d+1),\epsilon)) \subseteq C(-1/(d+1),c\epsilon)$. By Lemma~\ref{lemma: cone to cone} we also have that $f_{\lambda}(C(-1/(d+1),\epsilon)) \subseteq C(0, d \epsilon)$ and thus we can conclude that $f_{\lambda}$ maps $C(-\frac{1}{d+1},\epsilon)$ strictly into itself.

We now prove the claim. We calculate 
\[
    g_\lambda'(w) = -\frac{\lambda d(d+1) e^w}{(\frac{d}{d+1} + e^w)((\frac{d}{d+1}+e^w)^d+\lambda(d+1))}.
\]
This is a rational function in $e^w$, say $h(e^w)$, with the property that $h(0) = h(\infty) = 0$. If we take $w \in S_{\leq \delta}$ we can let $|e^w|$ get arbitrarily close to either $0$ or $\infty$ by adding the requirement that $|\Re(w)| \geq B$ for some large $B > 0$. From now on fix $B$ such that it guarantees that $|g_\lambda'(w)| < \frac{1}{2}$ if $w \in S_{\leq \delta}$ with $|\Re(w)| \geq B$.

It remains to be shown that there is a neighborhood of the real interval $[-B,B]$ on which $|g_\lambda'|$ is strictly less than $1$. To prove this we will show that there is a $\alpha \in (0,1)$ such that $|g_\lambda'(w)| \leq \alpha$ for all $w \in \mathbb{R}$. Note that $g_\lambda'(w) < 0$ for all real $w$. We calculate the second derivative.
\[
    g_{\lambda}''(w) = \frac{\lambda d^2(d+1) e^w \left[(e^w+\frac{d}{d+1})^d(e^w-\frac{1}{d+1}) - \lambda\right]}{(\frac{d}{d+1} + e^w)^2((\frac{d}{d+1}+e^w)^d+\lambda(d+1))^2}.
\]
The sign of $g_{\lambda}''(w)$ is the same as that of $(e^w+\frac{d}{d+1})^d(e^w-\frac{1}{d+1}) - \lambda$. This is a monotonically increasing function in $w$ that is negative as $w \to -\infty$ and positive as $w \to \infty$, therefore $g_{\lambda}''(w)$ has a unique sign change at some $w_m$. It follows that $g_\lambda'(w_m)$ is the global minimum of $g_\lambda'$ on the real line and that $w_m$ satisfies 
\[
    \lambda = \left(e^{w_m}+\frac{d}{d+1}\right)^d\left(e^{w_m}-\frac{1}{d+1}\right) = \left(\phi^{-1}(w_m)+1\right)^d \cdot \phi^{-1}(w).
\]

Observe that it follows that $\phi^{-1}(w_m)$ is a positive fixed point of $f_{\lambda}$ and thus $w_m$ is a fixed point of $g_{\lambda}$. Because $0< \lambda < \lambda_c$ the unique positive fixed point of $f_{\lambda}$ is attracting, i.e. $|f_{\lambda}'(\phi^{-1}(w_m))| < 1$; see Lemma~\ref{lem: lambdacritical}. The derivative of a fixed point is invariant under coordinate transformation and thus $|g'(w_m)| < 1$.

We conclude that for $\alpha := |g'(w_m)|$ we have $|g'(w)| < \alpha$ for all $w \in \mathbb{R}$. Because $[-B,B]$ is compact, there is an $\epsilon_0$ neighborhood of $[-B,B]$ such that $|g'(w)| < (1+\alpha)/2$ for all $w$ in this neighborhood. It follows that $|g'(w)| \leq (1+\alpha)/2$ for all $w \in S_{\leq \epsilon_0}$. This proves the claim and therefore concludes the proof of the lemma.
\end{proof}

We remark here that this shows that for $0 < \lambda < \lambda_c$ and $\epsilon$ sufficiently small $C(-1/(d+1),\epsilon)$ satisfies the third condition of Lemma~\ref{lem: properties of A}. Showing that it also satisfies the first is not difficult. This is sufficient to prove zero-freeness for graphs. Lemma~\ref{lem: Cone forward invariant} therefore essentially gives an alternative proof to that of \cite{PetersRegtsSokal} of the zero-freeness of a neighborhood of $(0,\lambda_c)$ for graphs. Unfortunately $\mu(C(-1/(d+1)),\epsilon)$ is not a multiplicative semigroup for small $\epsilon$ which is why we had to define $A(x,x_0,\epsilon)$.

\begin{lemma}
    \label{lem: forward invariance}
    Let $0 < \lambda < \lambda_c$ and $x_0 \geq \lambda - \frac{1}{d+1}$. For $\epsilon$ sufficiently small the region $A(-\frac{1}{d+1},x_0,\epsilon)$ is strictly forward invariant for $f_{\lambda}$.
\end{lemma}
\begin{proof}
Let $a_\epsilon$ and $b_\epsilon$ be the line-segments of the boundary of $A_\epsilon = A(-\frac{1}{d+1},x_0,\epsilon)$ in the upper-half plane with intersection $I_\epsilon$, with $a_\epsilon$ being the bounded line segment. Let us denote the imaginary part of $I_\epsilon$ with $h(\epsilon)$, that is \[h(\epsilon)=\Im(I_\epsilon)=\left(x_0+\frac{1}{d+1}\right)\tan(\epsilon)=\left(x_0+\frac{1}{d+1}\right)\epsilon+\mathcal{O}(\epsilon^2).\]

First we will prove that for sufficiently small $\epsilon$ the region $A_\epsilon$ is mapped into $S_{<h(\epsilon)}$. It is sufficient to prove that the line-segments $a_\epsilon,b_\epsilon$ are mapped into $S_{<h(\epsilon)}$.

The line-segment $b_\epsilon$ is part of the half-line defined by $-1+te^{\tilde{\epsilon}i}$ where $t\in[0,\infty)$. Under $f_\lambda$ it is injectively mapped into a straight line segment. Therefore $f_{\lambda}(b_\epsilon)$ is the straight line segment between $f_{\lambda}(I_\epsilon)$ and $f_\lambda(\infty) = 0$.
As $S_{<h(\epsilon)}$ is convex and contains 0, it is sufficient to show that  $f_\lambda(I_\epsilon)\in S_{<h(\epsilon)}$, i.e. what remains is to show that $f_\lambda(a_\epsilon)\subset S_{<h(\epsilon)}$.

Let us consider the line-segment $a_\epsilon$, that is parametrized by the curve $\gamma_\epsilon(t)=-(1-t)\frac{1}{d+1}+tI_\epsilon$, where $t\in[0,1]$.
Then
\begin{align*}
    f_\lambda(\gamma(t))=&\lambda\left(1-(1-t)\frac{1}{d+1}+tI_\epsilon\right)^{-d}\\
    =&\lambda\left(\frac{d}{d+1}+t(x_0+\frac{1}{d+1})\right)^{-d}-\lambda dt(x_0+\frac{1}{d+1})\left(\frac{d}{d+1}+t(x_0+\frac{1}{d+1})\right)^{-d-1}\epsilon i+\mathcal{O}(\epsilon^2),
\end{align*}
where the $\mathcal{O}(\epsilon^2)$ is uniform over all $t \in [0,1]$. It therefore sufficient to show that 
\begin{equation}
    \label{eq: inequality}
    \lambda dt(x_0+\frac{1}{d+1})\left(\frac{d}{d+1}+t(x_0+\frac{1}{d+1})\right)^{-d-1} < \lim_{\epsilon \to 0} h(\epsilon)/\epsilon = x_0 + \frac{1}{d+1}
\end{equation}
for all $t \in [0,1]$. We find 

\[
\frac{\lambda d t(x_0+\frac{1}{d+1})}{\left(\frac{d}{d+1}+t(x_0+\frac{1}{d+1})\right)^{d+1}} \leq 
\max_{\tilde{t}\ge 0}\frac{\lambda d\tilde{t}}{\left(\frac{d}{d+1}+\tilde{t}\right)^{d+1}}=\lambda \frac{d}{d+1}.\]
The inequality in (\ref{eq: inequality}) immediately follows and thus we conclude that $f_\lambda(A_\epsilon)\subseteq S_{<h(\epsilon)}$ if $\epsilon$ is sufficiently small.

To finish the proof observe that $A_\epsilon\subseteq C(-\frac{1}{d+1},\epsilon)$. Let $T$ be the closure of $f_\lambda(C(-\frac{1}{d+1},\epsilon))$, which by Lemma~\ref{lem: Cone forward invariant} is a proper subset of the interior of $C(-\frac{1}{d+1},\epsilon)$ for $\epsilon$ sufficiently small. Therefore, if $\epsilon$ is sufficiently small, then
\[
f_\lambda(A_\epsilon)\subseteq T \cap S_{< h(\epsilon)}
\]
is a proper subset of the interior of $A_\epsilon$.
\end{proof}

The condition in Lemma~\ref{lem: forward invariance} will unfortunately not be sufficient for $d=2$ and $d=3$. For those two cases we have a separate \emph{stronger} lemma using explicit values of $x_0$. The proof is very similar to the proof of the previous lemma.

\begin{lemma}
\label{lem: cases 2 3}
\hfill
    \begin{itemize}
    \item Let $\lambda \in (0,4)$. For $\epsilon$ sufficiently small the region $A(-\frac{1}{3},2,\epsilon)$ is strictly forward invariant for the map $z \mapsto \lambda/(1+z)^2$.
    \item
    Let $\lambda \in (0,\frac{27}{16})$. For $\epsilon$ sufficiently small the region $A(-\frac{1}{4},1,\epsilon)$ is strictly forward invariant for the map $z \mapsto \lambda/(1+z)^3$.
    \end{itemize}
\end{lemma}

\begin{proof}
Define $f_{\lambda,d}(z) = \lambda/(1+z)^d$, $A_{\epsilon,2} = A(-\frac{1}{3},2,\epsilon)$ and $A_{\epsilon,3} = A(-\frac{1}{4},1,\epsilon)$. Furthermore let $\tilde{\epsilon}_2 = \tan^{-1}(\frac{7}{9}\tan(\epsilon))$ and $\tilde{\epsilon}_3 = \tan^{-1}(\frac{5}{8}\tan(\epsilon))$. By definition we have 
\[
    A_{\epsilon,2} = C(-1/3, \epsilon) \cap C(-1,\tilde{\epsilon}_2)
    \quad
    \text{ and }
    \quad 
    A_{\epsilon,3} = C(-1/4, \epsilon) \cap C(-1,\tilde{\epsilon}_3).
\]

It follows from Lemma~\ref{lem: Cone forward invariant} that to prove forward invariance of $A_{\epsilon,d}$ it is sufficient to prove that $f_{\lambda,d}(A_{\epsilon,d})$ is a strict subset of $C(-1,\tilde{\epsilon}_d)$ for $\epsilon$ sufficiently small. The boundary of $A_{\epsilon,d}$ consists of two finite straight line segments, $a_{\epsilon,d}$ and $\overline{a_{\epsilon,d}}$ and two infinite line segments $b_{\epsilon,d}$ and $\overline{b_{\epsilon,d}}$. These line segments intersect in $I_{\epsilon,d}$ and $\overline{I_{\epsilon,d}}$, where $I_{\epsilon,2} = 2 + \frac{7}{3}\tan(\epsilon) i$ and $I_{\epsilon,3} = 1 + \frac{5}{4}\tan(\epsilon) i$. To prove that $A_{\epsilon,d}$ gets mapped into $C(-1,\tilde{\epsilon}_d)$ it is sufficient to prove that the segments $a_{\epsilon,d}$ and $b_{\epsilon,d}$ get mapped into $C(-1,\tilde{\epsilon}_d)$. Furthermore, because $f_{\lambda,d}(b_{\epsilon,d})$ is a straight line segment from $f_{\lambda,d}(I_\epsilon)$ to $0$, it is in fact sufficient to prove that $a_{\epsilon,d}$ is mapped strictly inside $C(-1,\tilde{\epsilon}_d)$ by $f_{\lambda,d}$.

We parameterize $a_{\epsilon,d}$ as $\gamma_{\epsilon,d}(t)$ for $t \in [0,1]$ where
\[
    \gamma_{\epsilon,2}(t) = -\frac{1}{3}(1-t) + (2 + \frac{7}{3} \tan(\epsilon) i)t
    \quad 
    \text{ and }
    \quad
    \gamma_{\epsilon,3}(t) = -\frac{1}{4}(1-t) + (1 + \frac{5}{4} \tan(\epsilon) i)t.
\]
Write $f_{\lambda,d}(\gamma_{\epsilon,d}(t)) = g_{\epsilon,d}(t) + i h_{\epsilon,d}(t)$, where $g_{\epsilon,d}$ and $h_{\epsilon,d}$ are real valued functions. We need to show that for $\epsilon$ sufficiently small $|h_{\epsilon,d}(t)/(g_{\epsilon,d}(t) + 1)| < \tan(\tilde{\epsilon}_d)$ for all $t \in [0,1]$. We can explicitly calculate\footnote{Which we did with the computer-algebra system \emph{Mathematica}.} that 
\[
    B_2(\lambda;t):=
    \lim_{\epsilon \to 0} \frac{h_{\epsilon,2}(t)}{(g_{\epsilon,2}(t) + 1)\tan(\tilde{\epsilon}_2)} = -\frac{162 t}{9 (7 t+2)+\frac{1}{\lambda}(7 t+2)^3}
\]
and 
\[
    B_3(\lambda;t):=
    \lim_{\epsilon \to 0} \frac{h_{\epsilon,3}(t)}{(g_{\epsilon,3}(t) + 1)\tan(\tilde{\epsilon}_3)} = -\frac{1536 t}{64 (5 t+3)+\frac{1}{\lambda}(5 t+3)^4}.
\]
Clearly $B_d(\lambda;t)$ is negative for all $t > 0$ and $|B_d(\lambda;t)|$ is increasing in $\lambda$. It therefore suffices to prove that $B_2(4;t) > -1$ and $B_3(\frac{27}{16};t) > -1$. We calculate explicitly 
\[
    \frac{d}{dt} B_2(4;t) = 
    \frac{1296 \left(49 t^2 (7 t+3)-40\right)}{(7 t+2)^2 (7 t (7 t+4)+40)^2}
    \quad\text{ and }\quad
    \frac{d}{dt} B_3(\frac{27}{16};t)
    =
    \frac{7776 \left(5 t^2 \left(25 t^2+40 t+18\right)-27\right)}{5 (5 t+3)^2 (25 t^3+45 t^2+27 t+27)^2}.
\]
Both these derivatives have a unique positive real zero and thus $B_2(4;t)$ and $B_3(\frac{27}{16};t)$ have a unique global minimum. These can be calculated numerically yielding $B_2(4,t) > -0.916$ and $B_3(\frac{27}{16},t) > -0.891$ for all $t\in [0,1]$. Because the interval $[0,1]$ is compact, it follows that the inequality $|h_{\epsilon,d}(t)/(g_{\epsilon,d}(t) + 1)| < \tan(\tilde{\epsilon}_d)$ holds for all $\epsilon$ sufficiently small, which concludes the proof.
\end{proof}

\subsection{Proof of the main theorem}
We are now ready to prove Theorem~\ref{thm: hypergraph Sokal region}, which we will restate here for convenience.
\theoremSokal*
\begin{proof}
    Let $d = \Delta - 1$ and fix $\lambda_0 \in (0,\lambda_c(d+1))$. We claim that we can choose $x_0(d)$ such that $\lambda_0$ and $A_{d,\epsilon}:= A(-\frac{1}{d+1},x_0(d),\epsilon)$ satisfy the properties of Lemma~\ref{lem: properties of A}. The theorem can then be concluded by Corollary~\ref{cor: properties imply zerofree}.

    By Lemma~\ref{lem: log convexity} the set $1 + A_{d,\epsilon}$ is log-convex for any choice of $x_{0}(d) \geq 0$ and $\epsilon \in (0,\pi/2)$. It follows from Corollary~\ref{cor: multiplicative semigroup} that $\mu(A_{d,\epsilon})$ is a multiplicative semigroup for small $\epsilon$ if 
    \begin{equation}
        \label{eq: inequality condition}
        \frac{1}{d^2-1} < x_0(d) < \frac{1}{\sqrt{d}-1}.
    \end{equation}
    
    If we choose $x_{0}(2) = 2$ and $x_0(3) = 1$ we see that inequality (\ref{eq: inequality condition}) is satisfied. Moreover, it follows from Lemma~\ref{lem: cases 2 3} that $A_{d,\epsilon}$ is strictly forward invariant for $f_{\lambda_0}$ for small $\epsilon$ in these two cases. This proves the claim for $d \in \{2,3\}$.

    Now we assume $d \geq 4$. It follows from Lemma~\ref{lem: forward invariance} that $A_{d,\epsilon}$ is strictly forward invariant for $f_{\lambda_0}$ for small $\epsilon$ if 
    \begin{equation}
        \label{eq: l - 1/(d+1) < x_0}
         \lambda_0 - \frac{1}{d+1} \leq x_0(d).
    \end{equation}
    We can thus prove the claim if we can show that we can simultaneously satisfy inequalities (\ref{eq: inequality condition}) and (\ref{eq: l - 1/(d+1) < x_0}). Because $\lambda_0 < \lambda_c(d+1)$ it is sufficient to show that for $d \geq 4$
    \[
        \frac{d^d}{(d-1)^{d+1}} - \frac{1}{d+1} < \frac{1}{\sqrt{d} - 1}.
    \]
    For $d = 4$ the inequality reads $\frac{1037}{1215} < 1$, which is true. Otherwise we observe that 
    \[
    \frac{d^d}{(d-1)^{d+1}} - \frac{1}{d+1}  = \frac{d}{(d-1)^2} \left(1 + \frac{1}{d-1}\right)^{d-1} - \frac{1}{d+1} < \frac{d}{(d-1)^2}e - \frac{1}{d+1}.
    \]
    It is thus sufficient to prove that for $d \geq 5$ 
    \begin{equation}
        \label{eq: final inequality}
        \frac{d}{(d-1)^2}e - \frac{1}{d+1} < \frac{1}{\sqrt{d} - 1}.
    \end{equation}
    As $d \to \infty$ the left-hand side of (\ref{eq: final inequality}) decreases as $1/d$ and the right-hand side as $1/\sqrt{d}$ and thus the inequality is certainly true for large $d$. Moreover, both sides are rational functions in $\sqrt{d}$ and thus they can be equal for only finitely many real values. We can accurately approximate the solutions to the resulting polynomial equation to find that the largest $d$ for which both sides are equal is at $d \approx 4.0389$. This shows that for $d \geq 5$ inequality (\ref{eq: final inequality}) is true, which concludes the proof of the claim.
\end{proof}

\section{The large degree limit}
\label{sec: limit region}
For every $d \geq 1$ we let $\mathcal{U}_{d,2}$ be the maximal open zero-free region for graphs of maximum degree at most $d+1$, formally
\[
\mathcal{U}_{d,2} = \mathbb{C}\setminus 
    \overline{\left\{\lambda \in \mathbb{C}: \text{there exists a graph $G$ with $\Delta(G) \leq d+1$ such that $Z(G;\lambda)=0$}\right\}}.
\]
We let $\mathcal{U}_{d,\geq 2}$ denote the analogous region for hypergraphs. To conform with notation from \cite{BencsBuys2021} we shifted the index by one with respect to the introduction. In the previous two sections we showed that for $d \geq 2$
\[
    \mathcal{U}_{d,2} \cap \mathbb{R} = \mathcal{U}_{d,\geq 2} \cap \mathbb{R} = \left(-\lambda_s(d+1), \lambda_c(d+1)\right).
\]
However, we will show in Lemma~\ref{cor: U3 is not zero-free} that $\mathcal{U}_{2,2} \neq \mathcal{U}_{2,\geq 2}$ and thus in general these sets are not equal. In this section we show that in the large degree limit they are in fact equal.

As $d$ increases the sets $\mathcal{U}_{d,2}$ become arbitrarily small, e.g. both $\lambda_c(d+1)$ and $\lambda_s(d+1)$ converge to $0$. However, as $d \to \infty$ it is not hard to see that $d \cdot \lambda_c(d+1)$ and $d \cdot \lambda_s(d+1)$ converge to $e$ and $1/e$ respectively. In fact, in \cite{BencsBuys2021} it is proved that there is a set $\mathcal{U}_\infty$ such that the sequence of rescaled regions $d \cdot \mathcal{U}_{d,2}$ converges to $\mathcal{U}_\infty$. This limit set is bounded and contains an open neighborhood of $(-1/e,e)$.
\begin{theorem}[Theorem 1.1. in \cite{BencsBuys2021}]
\label{thm: limit theorem graphs}
The sets $d \cdot \mathcal{U}_{d,2}$ converge to $\mathcal{U}_\infty$ in terms of the Hausdorff distance.
\end{theorem}
This means that for every closed $K_1 \subseteq \Int(\mathcal{U}_\infty)$ and open $K_2 \supseteq \overline{\mathcal{U}_\infty}$ for $d$ sufficiently large $K_1 \subseteq d \cdot \mathcal{U}_{d,2} \subseteq K_2$. We will show that the same is true for the rescaled regions $d \cdot \mathcal{U}_{d, \geq 2}$.

To show this we define the map $E_\Lambda(Z) = \Lambda e^{-Z}$. It is shown in \cite{BencsBuys2021} that $\Lambda \in \mathcal{U}_\infty$ if and only if there is a compact convex set containing $0$ that is forward invariant for $E_{\Lambda}$. We will show that, after doing the proper rescaling, this implies that there is a set $A$ such that $\Lambda/d$ and $\tilde{A}/d$ satisfy the conditions of Lemma~\ref{lem: properties of A} for sufficiently large $d$. More precisely, we prove the following.

\begin{lemma}
\label{lem: rescaling conditions}
Let $\Lambda_0 \in \mathbb{C}$ and $A \subseteq \mathbb{C}$ convex, compact and with $0 \in A$ such that $E_{\Lambda_0}(A) \subseteq \Int(A)$. Then there is a neighborhood $U$ of $\Lambda_0$, a region $\tilde{A}$ and a $d_0 \in \mathbb{Z}_{\geq 1}$ such that for all $\Lambda \in U$ and $d \geq d_0$ the pair $\Lambda/d$ and $\tilde{A}/d$ satisfy the conditions of Lemma~\ref{lem: properties of A}, i.e. 
\begin{itemize}
    \item $1 + \tilde{A}/d$ is log-convex;
    \item $\mu(\tilde{A}/d)$ is a multiplicative semi-group;
    \item $\tilde{A}/d$ is strictly forward invariant for $z \mapsto \frac{\Lambda/d}{(1+z)^d}$.
\end{itemize}
\end{lemma}

\begin{proof}
We need to slightly alter $A$ to define $\tilde{A}$. For $\epsilon > 0$ let $A_\epsilon$ denote the convex hull of the union of $A$ with the closed ball $\B{\epsilon}$. For $\epsilon$ sufficiently small we still have $E_{\Lambda_0}(A_\epsilon) \subseteq \Int(A_\epsilon)$; from now on fix such an $\epsilon$.

We recall that $\mathbb{D}$ denotes the open unit disk. By the Riemann mapping theorem there is a conformal bijection $h: \mathbb{D} \to \Int(A_\epsilon)$. Let $\delta < 1$ be sufficiently close to $1$ such that for $\tilde{A} = h(\B{\delta})$ we still have $0 \in \Int(\tilde{A})$ and $E_{\Lambda_0}(\tilde{A}) \subseteq \Int(\tilde{A})$. We will now show that $\tilde{A}$ satisfies the required properties. 

Observe that, because $\tilde{A}$ is compact, for $d$ sufficiently large $1+ \tilde{A}/d$ is contained in the right half-plane and thus $\log(1+ \tilde{A}/d)$ is well-defined. The interior of $\log(1+ \tilde{A}/d)$ is the image of $\mathbb{D}$ under the map $g(z) = \log(1+h(\delta z)/d)$. It is a result from complex analysis (see e.g.~\cite[$\S 2.5$]{Duren1983}) that the image $g(\mathbb{D})$ of a conformal map $g$ is convex if and only if for all $z \in \mathbb{D}$ we have $\Re(1 + z g''(z)/g'(z)) > 0$. We apply this to $g$:
\begin{align*}
    \inf_{z \in \mathbb{D}}\Re\left(1 + \frac{z g''(z)}{g'(z)}\right) =  \inf_{z \in \mathbb{D}}\Re\left(1 + \frac{\delta z h''(\delta z)}{h'(\delta z)} + \frac{\delta z h'(\delta z)}{h(\delta z) + d}\right)
    = \min_{z \in \B{\delta}} \Re\left(1 + \frac{z h''(z)}{h'(z)} + \frac{zh'(z)}{h(z) + d}\right).
\end{align*}
Because $h(\mathbb{D}) = A_\epsilon$ is convex there is a strictly positive lower bound on $\Re(1 + z h''(z)/h'(z))$ for $z \in \B{\delta}$. For $d$ sufficiently large, say $d \geq d_1$, it thus follows that the above infimum is strictly positive and thus $g(\mathbb{D}) = \log(1+ \tilde{A}/d)$ is convex.

Because $\tilde{A}$ is compact and contains $0$ in its interior there are $0<m<M$ such that $\B{m}\subseteq \tilde{A}\subseteq \B{M}$. For $0< r< 1$ we have $\B{r/(1+r)} \subset \mu(\B{r}) \subset \B{r/(1-r)}$. Thus, for $d$ sufficiently large, 
\[
    \B{m/(d+m)} \subseteq \mu(\tilde{A}/d) \subseteq \B{M/(d-M)}.
\]
For $w_1,w_2 \in \tilde{A}/d$ we have 
\[
    |w_1 w_2| \leq \left(\frac{M}{d-M}\right)^2,
\]
which, for say $d \geq d_2$, is less than $\frac{m}{d+m}$. We can conclude that $w_1 w_2 \in \mu(\tilde{A}/d)$ and thus $\mu(\tilde{A}/d)$ is a multiplicative semi-group for $d \geq d_2$.

Let $G_d(\Lambda,Z) = \frac{\Lambda}{(1+Z/d)^d}$. The set $\tilde{A}/d$ being strictly forward invariant for $z \mapsto \frac{\Lambda/d}{(1+z)^d}$ is equivalent to saying that $\tilde{A}$ is strictly forward invariant for $Z \mapsto G_d(\Lambda,Z)$. Let $K$ be a compact set containing $\Lambda_0$ in its interior. The maps $(\Lambda,Z) \to G_d(\Lambda,Z)$ converge uniformly on $K \times \tilde{A}$ to $(\Lambda,Z) \mapsto \Lambda e^{-Z}$. Therefore there is a neighborhood $U$ of $\Lambda_0$ and a $d_3$ such that for $
\Lambda \in U$ and $d \geq d_3$ we have that $G_d(\Lambda, \tilde{A}) \subseteq \Int(\tilde{A})$. We conclude the proof by letting $d_0 = \max\{d_1,d_2,d_3\}$.
\end{proof}

This allows us to prove Theorem~\ref{thm: limit theorem}, which we state in the following form.
\begin{theorem}
    The sets $d \cdot \mathcal{U}_{d,\geq 2}$ converge to $\mathcal{U}_\infty$ in terms of the Hausdorff distance.
\end{theorem}
\begin{proof}
Take a closed $K_1 \subseteq \Int(\mathcal{U}_\infty)$ and open $K_2 \supseteq \overline{\mathcal{U}_\infty}$. 

It follows from Theorem~\ref{thm: limit theorem graphs} that for $d$ sufficiently large $d \cdot \mathcal{U}_{d,2} \subseteq K_2$. Because $\mathcal{U}_{d, \geq 2} \subseteq \mathcal{U}_{d, 2}$ by definition, the same is true for $\mathcal{U}_{d, \geq 2}$.

Take a $\Lambda_0 \in K_1$. It follows from \cite[Corollary 4.2 and Lemma 4.4]{BencsBuys2021} that there is a compact convex set containing $0$ that is strictly forward invariant for $E_{\Lambda_0}$. It follows from Lemma~\ref{lem: rescaling conditions} that there is a neighborhood $U(\Lambda_0)$ of $\Lambda_0$, a region $\tilde{A}(\Lambda_0)$ and a $d_0(\Lambda_0)$ such that for $\Lambda \in U(\Lambda_0)$ and $d \geq d_0(\Lambda_0)$ the pair $\Lambda/d$ and $\tilde{A}(\Lambda)/d$ satisfy the conditions of Lemma~\ref{lem: properties of A}. It follows then from Corollary~\ref{cor: properties imply zerofree} that for those $d$ we have $\Lambda/d \in \mathcal{U}_{d,\geq 2}$, i.e. $U(\Lambda_0) \subseteq d \cdot \mathcal{U}_{d,\geq 2}$.

The sets $\{U(\Lambda)\}_{\Lambda \in K_1}$ form an open cover of $K_1$. Because $\mathcal{U}_\infty$ is bounded the set $K_1$ is compact and thus there is a finite $I \subseteq K_1$ such that $\{U(\Lambda)\}_{\Lambda \in I}$ is a cover of $K_1$. Then for $d \geq \max_{\Lambda \in I} d_0(\Lambda)$ we can conclude that $K_1 \subseteq d \cdot \mathcal{U}_{d,\geq 2}$.
\end{proof}

\section{The maximal zero-free disk for bounded degree uniform linear hypertrees}\label{sec: disk for linear hypertrees}

In this section we will prove Theorem~\ref{thm: Zero free linear (k)-uniform hypertrees}. Because $b,d$ will be more or less fixed we will drop the subscripts and write 
\begin{equation}
    \label{eq: flambda hypertrees}
    f_{\lambda}(z) = \lambda \cdot \left[1 - \left(\frac{z}{1+z}\right)^b \right]^d
\end{equation}
for the remainder of this section.


\subsection{Preliminaries on complex dynamics}
In this section we gather the required background on the theory of dynamics of rational maps. The definitions and general results from this section can be found in \cite{MilnorDynamics,McMullen1994}. We recall that $\Chat = \mathbb{C} \cup \{\infty\}$ denotes the Riemann sphere.  

\subsubsection{Critical points and holomorphic families of rational maps}
A point $z \in \Chat$ is called a \emph{critical point} of a rational map $g: \Chat \to \Chat$ if $g$ is not locally injective around $z$, i.e. if there does not exist a neighborhood $U$ of $z$ such that $g|_{U}$ is injective. A \emph{holomorphic family of rational maps} parameterized by an open set $X \subseteq \Chat$ is a holomorphic map $g: X \times \Chat \to \Chat$ such that for any fixed $\lambda_0 \in X$ the map $z \mapsto g(\lambda_0,z)$ is rational. We usually denote $g(\lambda,z)$ by $g_\lambda(z)$. Note that $f$ as defined in (\ref{eq: flambda hypertrees}) is a holomorphic family of rational maps for any open $X \subseteq \Chat$. We say that the critical points of a holomorphic family of rational maps \emph{move holomorphically} around $\lambda_0$ if there is a neighborhood $U \subset X$ of $\lambda_0$ and holomorphic maps $c_{1}, \dots, c_N$ from $U \to \Chat$ such that for any $\lambda \in U$ the critical points of $g_\lambda$ are $\{c_1(\lambda), \dots, c_N(\lambda)\}$.

For any map $g: \Chat \to \Chat$ we denote by $g^n$ the map that applies $g$ successively $n$ times. The set $\{g^n(z_0)\}_{n \geq 1}$ is called the \emph{forward orbit} of $z_0$. We shall see that the forward orbits of the critical points play an important role in understanding the dynamic behaviour of $g$. We first prove a lemma about the critical orbits of our map $f$ (as defined in (\ref{eq: flambda hypertrees})). The lemma says that there is essentially a unique critical orbit.

\begin{lemma}
    \label{lem: critical orbits}
    Around any $\lambda_0 \in \mathbb{C}\setminus\{0\}$ the critical points of $f_{\lambda}$ move holomorphically. In fact there are holomorphic functions $c_1, \dots, c_N$ parameterizing the critical points on $\mathbb{C}\setminus\{0\}$. Moreover, for any critical point $c_i$
    \[
        f^{n_i}_{\lambda}(c_i(\lambda)) = \lambda
    \]
    for some $n_i \in \{1,2,3\}$.
\end{lemma}

\begin{proof}

The maps $z \mapsto \lambda z$, $z\mapsto 1-z$ and $z \mapsto z/(1+z)$ are locally injective on the whole Riemann sphere. The maps $z \mapsto z^b$ and $z\mapsto z^d$ are locally injective anywhere except for $z \in \{0,\infty\}$. Therefore the critical points of $f_\lambda$, as long as $\lambda$ is nonzero, are given by those $z$ for which either
\[
    \frac{z}{1+z} \in \{0, \infty\} 
    \quad 
    \text{ or }
    \quad 
    1-\left(\frac{z}{1+z}\right)^b \in \{0, \infty\}.
\]
Since these do not depend on $\lambda$ we can conclude that the critical points do indeed move holomorphically on $\mathbb{C}\setminus\{0\}$; they are in fact constant. Moreover observe that
\begin{itemize}
    \item
    if $\frac{z}{1+z} = 0$ then $f_\lambda(z) = \lambda$;
    \item
    if $1-\left(\frac{z}{1+z}\right)^b = 0$ then $f^2_\lambda(z) = \lambda$;
    \item 
    if $\frac{z}{1+z} = \infty$, equivalently if $1-\left(\frac{z}{1+z}\right)^b = \infty$, then $f^3_{\lambda}(z) = \lambda$.
\end{itemize}
\end{proof}

\subsubsection{Normality}
Let $U$ be an open subset of $\Chat$ and let $\mathcal{F}$ be a set of holomorphic maps $g: U \to \Chat$. In this context $\mathcal{F}$ is usually referred to as a family of holomorphic maps. The family $\mathcal{F}$ is called a \emph{normal family} if every sequence $\{g_n\}_{n \geq 1} \subseteq \mathcal{F}$ has a subsequence that converges uniformly on compact subsets of $U$. Given a particular parameter $\lambda_0 \in U$ we say that $\mathcal{F}$ is \emph{normal at $\lambda_0$} if there exists an open neighborhood of $\lambda_0$ on which $\mathcal{F}$ is a normal family.

\begin{theorem}[Montel's Theorem]
\label{thm: Montel's theorem}
Let $\mathcal{F}$ be a family of holomorphic maps $U \mapsto \Chat$ for some open $U \subseteq \Chat$ and let $h_1,h_2,h_3$ be holomorphic maps $U \to \Chat$ such that $h_1(\lambda)$, $h_2(\lambda)$ and $h_3(\lambda)$ are pairwise distinct for all $\lambda \in U$. If $g(\lambda) \not \in \{h_1(\lambda),h_2(\lambda),h_3(\lambda)\}$ for all $g \in \mathcal{F}$ and $\lambda \in U$ then $\mathcal{F}$ is a normal family.
\end{theorem}

\begin{remark}
This theorem is usually stated with $h_1,h_2$ and $h_3$ being constant, say $0,1$ and $\infty$. By considering the family $\mathcal{F}' = \{\lambda \mapsto (\mu_\lambda \circ g)(\lambda): g \in \mathcal{F}\}$, where $\mu_\lambda$ is a M\"obius tranformation sending $(h_1(\lambda),h_2(\lambda),h_3(\lambda))$ to $(0,1,\infty)$, we see that $\mathcal{F}$ is a normal family if and only if $\mathcal{F}'$ is a normal family. Furthermore $\mathcal{F}$ avoids $\{h_1(\lambda),h_2(\lambda),h_3(\lambda)\}$ if and only if $\mathcal{F}'$ avoids $\{0,1,\infty\}$. The two statements are therefore equivalent.
\end{remark}

\subsubsection{Periodic points}

A point $z_0 \in \Chat$ is called a \emph{periodic fixed point} of a rational map $g: \Chat \to \Chat$ if there is an $n \geq 1$ such that $g^n(z_0)=z_0$. The least $n$ for which this is the case is called the \emph{period} of $z_0$. If $n=1$ then $z_0$ is called a \emph{fixed point} of $g$. The \emph{multiplier} of a periodic point $z_0$ with period $n$ is defined as $(g^n)'(z_0)$, i.e. the derivative of $f^n$ evaluated at $z_0$.\footnote{If the orbit of $z_0$ passes through $\infty$ the multiplier can be calculated by conjugating $g$ with a chart that moves the orbit to $\mathbb{C}$.} 

\begin{lemma}
    \label{lem: fixed points flambda}
    The map $f_{\lambda}$ has a fixed point with multiplier $\alpha$ if and only if $\lambda = (1-w^b)^{-d}\cdot\frac{w}{1-w}$, where $w \neq 1$ is a solution to 
    \[
        -bd \cdot \frac{w^b(1-w)}{1-w^b} = \alpha.
    \]
\end{lemma}

\begin{proof}
Note that $z$ is a fixed point of $f_{\lambda}$ if and only if 
\[
f_{\lambda}(z) = z
\quad
\iff
\quad 
\lambda = z \cdot \left[1 - \left(\frac{z}{1+z}\right)^b \right]^{-d}.
\]
Furthermore, this fixed point has multiplier $\alpha$ if and only if 
\[
    \alpha = f_\lambda'(z) =  -bd \frac{z^{b-1}}{(1+z)
   \left((1+z)^b-z^b\right)}\cdot f_{\lambda}(z) = -db\frac{\left(\frac{z}{1+z}\right)^b}{(1+z)\left(1-\left(\frac{z}{1+z}\right)^b\right)}.
\]
Note that $f_{\lambda}(\infty) = 0$ irregardless of $\lambda$ and thus $\infty$ is never a fixed point. Therefore, for any fixed point $z$, we can write $z = \frac{w}{1-w}$ for some $w \neq -1$. The result follows by making this substitution. 
\end{proof}



A periodic fixed point with multiplier $\alpha$ is called either \emph{attracting}, \emph{indifferent} or \emph{repelling} according to whether $|\alpha| < 1$, $|\alpha| = 1$ or $|\alpha| > 1$ respectively. A periodic point $z_0$ of $g_{\lambda_0}$, where $g$ is a holomorphic family of rational maps, is called \emph{persistently indifferent} if there is a neighborhood $U$ of $\lambda_0$ and a holomorphic map $w: U \to \Chat$ such that $w(\lambda_0) = z_0$ and $w(\lambda)$ is an indifferent periodic fixed point of $g_\lambda$ for every $\lambda \in U$.

\begin{lemma}
    \label{lem: no pers. indiff.}
    The map $f_\lambda$ has no persistently indifferent periodic points for any $\lambda \in \mathbb{C}$.
\end{lemma}

\begin{proof}
It follows from Lemma~\ref{lem: fixed points flambda} that there is a non-empty open set $U \subseteq \mathbb{C}$ such that $f_{\lambda}$ has an attracting fixed point for all $\lambda \in U$. Fix any nonzero $\lambda_0 \in U$, let $p$ denote an attracting fixed point of $f_{\lambda_0}$, and let $c_1, \dots, c_N$ denote the critical points of $f_{\lambda_0}$. At least one of the critical orbits $\{f^n_{\lambda_0}(c_i)\}_{n \geq 1}$ converges to $p$; see e.g. \cite[Theorem 8.6]{MilnorDynamics}. It then follows from Lemma~\ref{lem: critical orbits} that the orbit of any critical point converges to $p$. This means that $f_{\lambda_0}$ is a hyperbolic map; see e.g. \cite[\S 19]{MilnorDynamics}. One property of hyperbolic maps is that every periodic orbit must be either attracting or repelling, i.e. not indifferent. It is not hard to see that if $f_\lambda$ were to have a persistently indifferent fixed point for some $\lambda$, then it has to have an indifferent fixed point for all but finitely many $\lambda \in \mathbb{C}$; this is proved e.g. in \cite[Lemma 9]{buys_cayleytrees}. Because we just observed that $f_\lambda$ is hyperbolic for $\lambda$ in an open set, we can conclude that $f_{\lambda}$ cannot have a persistently indifferent fixed point for any $\lambda$.
\end{proof}

\begin{theorem}[Part of Theorem 4.2 in \cite{McMullen1994}]
			\label{thm: Hol Motion}
			Let $g$ be a holomorphic family of rational maps parameterized by $U$. Suppose 
			there exist holomorphic maps $c_i: U \to \Chat$ parameterizing the critical points of 
			$g$. Let $\lambda_0 \in U$ and suppose that for all $i$ the families of maps given by
				\[
					\mathcal{F}_i = \{\lambda \mapsto g^n_\lambda(c_i(\lambda))\}_{n \geq 1}
				\]
			are normal at $\lambda_0$. Then there is a neighborhood $V \subseteq U$ of $\lambda_0$ such that 
			for all $\lambda \in V$ every periodic point of $g_\lambda$ is either attracting, repelling or 
			persistently indifferent.
		\end{theorem}

\subsection{Accumulation of roots}
We will use the results gathered in the previous section to show that parameters $\lambda$ for which $f_{\lambda}$ (as defined in (\ref{eq: flambda hypertrees})) has an indifferent fixed point are accumulation points of zeros.

\begin{lemma}
    \label{lem: ratio -1 implies zero}
    Suppose $\lambda_0 \in \mathbb{C}$ and $n \in \mathbb{Z}_{\geq 1}$ are such that $f_{\lambda_0}^n(\lambda_0) = -1$ then there is a $(b+1)$-uniform linear hypertree $\mathcal{T}$ with degree at most $d+1$ and $Z_\mathcal{T}(\lambda_0) = 0$. 
\end{lemma}

\begin{proof}
We inductively define a sequence $\{(\mathcal{T}_m,v_m)\}_{m \geq 0}$ of rooted $(b+1)$-uniform linear hypertrees whose maximum degrees are at most $d+1$. We let $\mathcal{T}_0$ consist of a single vertex $v_0$. For $m \geq 1$ we let $\mathcal{T}_m$ consist of the vertex $v_m$ that is contained in $d$ hyperedges of size $b+1$ only intersecting in $v_m$. For all vertices $u \neq v_m$ sharing a hyperedge with $v_m$ we attach a disjoint copy of $(\mathcal{T}_{m-1},v_{m-1})$ by identifying $u$ with $v_{m-1}$.

It follows from Lemma~\ref{lem:tree_recursion} that
\[
    R_{v_m}(\mathcal{T}_m;\lambda) = f_\lambda^m(\lambda)
\]
for all $m \geq 0$ and $\lambda \in \mathbb{C}$.
Therefore $R_{v_n}(\mathcal{T}_n;\lambda_0) = -1$ and thus it follows from Lemma~\ref{lem:ratio_vs_zero} that there is a $(b+1)$-uniform linear hypertree $\mathcal{T}$ with degree at most $d+1$ and $Z_\mathcal{T}(\lambda_0) = 0$ (in fact, one can take $\mathcal{T} = \mathcal{T}_n$). 
\end{proof}

We remark that for graphs, i.e. for $b=1$, the trees defined in the proof of the lemma above are often refered to as $d$-ary trees of finite depth.

\begin{lemma}
    \label{lem: indiff implies zero}
    Let $\lambda_0 \in \mathbb{C}\setminus\{-1,0\}$ and suppose $f_{\lambda_0}$ has an indifferent fixed point. Then, given a neighborhood $U$ of $\lambda_0$, there is a $\lambda_1 \in U$ and a $(b+1)$-uniform linear hypertree $\mathcal{T}$ with degree at most $d+1$ such that $Z_\mathcal{T}(\lambda_1) = 0$.
\end{lemma}

\begin{proof}
Take $p_0 \in \mathbb{C}$ such that $f_{\lambda_0}(p_0) = -1$. It follows that $f_{\lambda_0}^2(p_0) = \infty$ and $f_{\lambda_0}^3(p_0) = 0$ and thus because $\lambda_0 \not\in\{-1,0,\infty\}$, it follows from Lemma~\ref{lem: critical orbits} that $p_0$ is not a critical point of $f_{\lambda_0}$. We can therefore apply the Implicit Function Theorem to obtain a neighborhood $V$ of $\lambda_0$ and a holomorphic map $p: V \to \Chat$ such that $p(\lambda_0) = p_0$ and $f_{\lambda}(p(\lambda)) = -1$ for all $\lambda \in V$. Note that this implies that $p(\lambda) \not \in \{-1,\infty\}$ for all $\lambda \in V$.

Let $U \subseteq V$ be an arbitrary neighborhood of $\lambda_0$. By Lemma~\ref{lem: critical orbits} there are holomorphic functions $c_1, \dots, c_N$ parameterizing the critical points of $f_\lambda$ on $U$. Because $f_{\lambda_0}$ has an indifferent fixed point that, by Lemma~\ref{lem: no pers. indiff.}, is not persistently indifferent it follows from Theorem~\ref{thm: Hol Motion} that at least one of the families $\mathcal{F}_i = \{\lambda \mapsto f^n_\lambda(c_i(\lambda))\}_{n \geq 1}$ is not normal on $U$. It then follows from Lemma~\ref{lem: critical orbits} that the family 
$\{\lambda \mapsto f^n_\lambda(\lambda)\}_{n \geq 1}$
is not normal on $U$. By Theorem~\ref{thm: Montel's theorem} there must be a $\lambda_1 \in U$ such that $f_{\lambda_1}^n(\lambda_1) \in \{-1,\infty,p(\lambda_1)\}$. If $f_{\lambda_1}^n(\lambda_1) = \infty$ then $f_{\lambda_1}^{n-1}(\lambda_1) = -1$ and if $f_{\lambda_1}^n(\lambda_1) = p(\lambda_1)$ then $f_{\lambda_1}^{n+1}(\lambda_1) = -1$. In all three cases we can apply Lemma~\ref{lem: ratio -1 implies zero} to conclude that there is a $(b+1)$-uniform linear hypertree $\mathcal{T}$ with degree at most $d+1$ such that $Z_\mathcal{T}(\lambda_1) = 0$.
\end{proof}

Using Lemma~\ref{lem: fixed points flambda} we can calculate parameters $\lambda$ for which $f_{\lambda}$ has an indifferent fixed point. Figure~\ref{fig: indiff_par} shows examples of such parameters, which, by the previous lemma, are accumulation points of zeros of $(b+1)$-uniform hypertrees.

Let $\mathcal{U}_{\Delta,2}$ be the maximal connected zero-free region containing $0$ for graphs of maximum degree at most $\Delta$ (shifting the index back from the one used in Section~\ref{sec: limit region} to the one in the introduction). In \cite[Conjecture 6]{PerkinsHypergraphs} it is conjectured that this region is zero-free for hypergraphs of maximum degree at most $\Delta$ as well. Using Lemma~\ref{lem: indiff implies zero} we can show that, at least for $\Delta = 3$, this is not the case.

\begin{corollary} \label{cor: U3 is not zero-free}
    The region $\mathcal{U}_{3,2}$ is not zero-free for hypergraphs of maximum degree at most $3$.
\end{corollary}

\begin{proof}
    Let $\lambda_0 = (1-w_0^4)^{-2}\frac{w_0}{1-w_0}$, where $w_0 \approx 0.3540 + 0.5331 i$ is a solution to
    \[
        -8 \frac{w_0^4(1-w_0)}{(1-w_0^4)}=1.
    \]
    We find that $\lambda_0 \approx 0.0665 + 0.6015 i$; see also Figure~\ref{fig: indiff_par}. By Lemmas~\ref{lem: fixed points flambda}~and~\ref{lem: indiff implies zero} zeros of $5$-uniform linear hypertrees with maximum degree at most $3$ accumulate on $\lambda_0$.

    On the other hand, it is shown in \cite{bencs2023complex} that $\mathcal{U}_{d+1,2}$ contains the intersection of the disc centered at $0$ of radius $\frac{7}{8}\tan(\frac{\pi}{2d})$ and the right half-plane. Because $\frac{7}{8}\tan(\frac{\pi}{4}) = 0.875$ and $|\lambda_0| \approx 0.6052$ we can conclude that $\lambda_0 \in \mathcal{U}_{3,2}$.
\end{proof}


\begin{figure}[h!]
\label{fig: indiff_par}
\center
\includegraphics{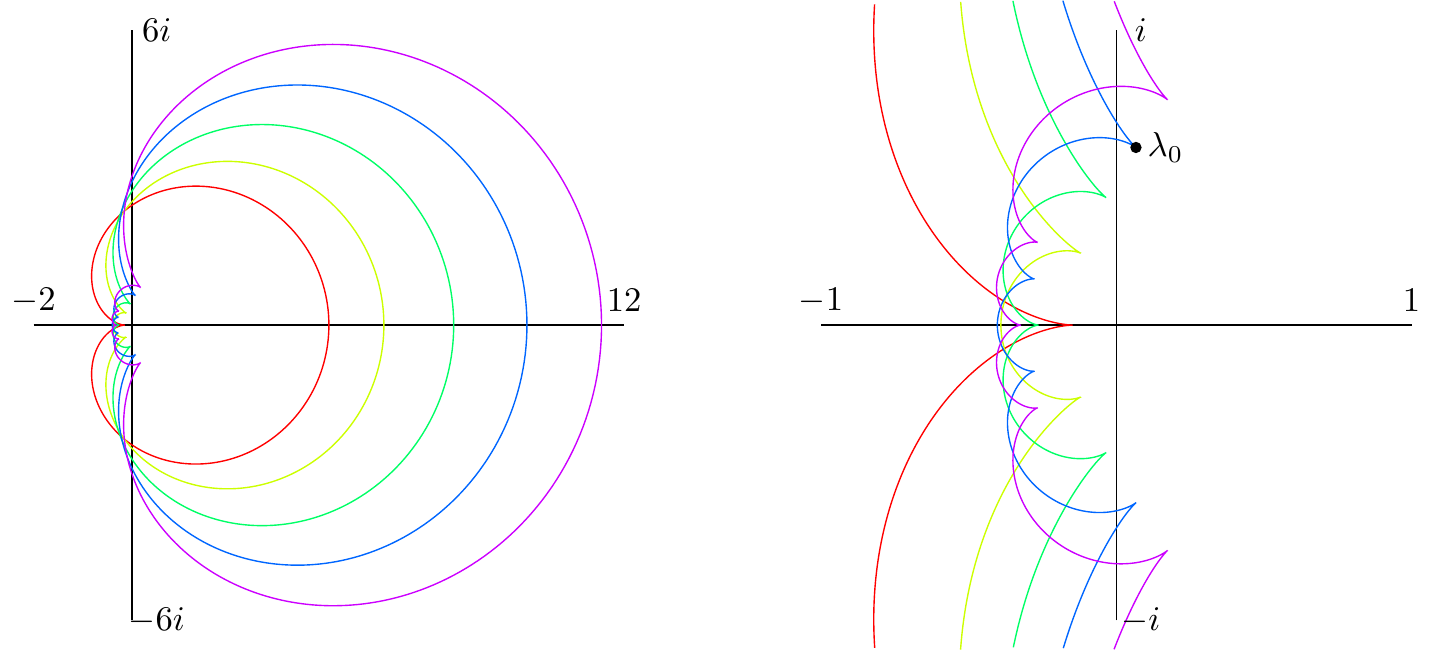}
\caption{Parameters $\lambda$ for which $f_{\lambda,b}$ has an indifferent fixed point for $d=2$ and $b$ ranging from $1$ to $5$. The regions enclosed by the curves get increasingly larger.}
\end{figure}

\subsection{Dynamics on disks}

To prove the zero-freeness part of Theorem~\ref{thm: Zero free linear (k)-uniform hypertrees} we will apply Lemma~\ref{lem: zero-free disk k-uniform hypertrees}. We will thus show that there is a $\lambda$ with $|\lambda| = \rho_{d,b}$ such that $f_{\lambda}(B_R) \subseteq B_R$ for some $R > 0$. Instead of using the value of $\rho_{d,b}$ we will do this by letting $\lambda_0$ be maximal in norm for which we can apply Lemma~\ref{lem: zero-free disk k-uniform hypertrees}. It will follow that there must be a $\lambda$ of the same norm as $\lambda_0$ for which $f_{\lambda}$ has a fixed point of multiplier $1$. We must first prove some lemmas about disks.

\begin{lemma}
\label{lem: Two maxima}
    Let $D\subset \mathbb{C}$ be a closed disk with $0 \in \Int(D)$ whose center lies in $\mathbb{R}_{< 0}$ and let $z_0$ be the intersection of $\partial D$ with the negative real axis. Then 
    \begin{itemize}
        \item for odd $n \in \mathbb{Z}_{\geq 1}$
        \[
            \argmax_{z \in D} |z^n-1| = \{z_0\};
        \]
        \item
        for even $n \in \mathbb{Z}_{\geq 1}$ either
        \[
            \argmax_{z \in D} |z^n-1| = \{z_0\}
            \quad
            \text{ or }
            \quad 
            \argmax_{z \in D} |z^n-1| = \{z_M, \overline{z_M}\},
        \]
        where $z_M$ is a non-real value in $\partial D$.
    \end{itemize}
\end{lemma}

\begin{proof}
Assume first that $n$ is odd. Let $z_0$ be the intersection of $\partial D$ with the negative real axis. For any $r > 0$ let $\B{r}$ denote the closed disk centered at $0$ with radius $r$. For a set $A$ let $A^n = \{z^n: z\in A\}$. Observe that $D \subseteq \B{|z_0|}$ and therefore $D^n \subseteq \B{|z_0|}^n = \B{|z_0|^n}$. The unique maximizer of $|z-1|$ with $z \in \B{|z_0|^n}$ is $-|z_0|^n = z_0^n$. Because $D\setminus\{z_0\} \subseteq \Int(\B{|z_0|})$ it follows that $z_0$ is the unique maximizer for $|z^n-1|$ with $z \in D$.

Now assume that $n$ is even. Let $-x$ and $r$ be the center and the radius of $D$ respectively and note that $x > r$. The disk $D$ is the image of the closed unit disk $\overline{\mathbb{D}}$ under the map $z \mapsto rz - x$. Because $n$ is even and $\overline{\mathbb{D}}$ is symmetric under $z \mapsto -z$ we see that $D^n$ is the image of $\overline{\mathbb{D}}$ under the map $p(z):=(rz + x)^n$. We want to determine the values in $\overline{\mathbb{D}}$ that maximize $|p(z) - 1|$. These values have to lie on the boundary $\partial \mathbb{D}$. 

We define $\gamma(t) = p(e^{\pi i t})$ for $t \in [0,1]$, which parameterizes the image under $p$ of the part of $\partial \mathbb{D}$ that is contained in the upper half-plane. It is sufficient to show that there is a unique $t_0 \in [0,1)$ maximizing $|\gamma(t) - 1|$. Note that either $0 < \gamma(1) \leq 1$ or $1 \leq \gamma(1) < \gamma(0)$ so $t=1$ is never a maximizer of $|\gamma(t) - 1|$.

The curvature of the curve $\gamma$ is given by 
\[
    \kappa(t) = \frac{1}{|p'(e^{\pi i t})|} \Re\left(1 + \frac{e^{\pi i t} p''(e^{\pi i t})}{p'(e^{\pi i t})}\right).
\]
We will show that $\kappa$ is strictly increasing.

We have that $p'(e^{\pi i t}) = nr(re^{\pi i t}+x)^{n-1}$. As $t$ increases $re^{\pi i t}+x$ travels along a circle of radius $r$ centered at $x$ from the positive real axis to the negative real axis. Since $x > 0$ we see that $|re^{\pi i t}+x|$ is decreasing and therefore $1/|p'(e^{\pi i t})|$ is increasing. 

We calculate 
\[
1 + \frac{e^{\pi i t} p''(e^{\pi i t})}{p'(e^{\pi i t})} = 1 + \frac{e^{\pi i t} (n-1)r}{re^{\pi i t} + x} = \frac{nr e^{\pi i t} + x}{r e^{\pi i t} + x}.
\]
The map $z \mapsto (nrz + x)/(rz + x)$ is a M\"obius transformation that maps the real line to the real line. Because a M\"obius transformation is conformal and sends generalized circles to generalized circles it must map $\partial\mathbb{D}$ to a circle crossing the real line transversely at $(nr + x)/(r + x)$ and $(-nr+x)/(-r + x)$. Observe that 
\[
    0 < \frac{nr + x}{r + x} < \frac{-nr+x}{-r + x}
\]
which implies that as $t$ increases from $0$ to $1$ the real part of $(nr e^{\pi i t} + x)/(r e^{\pi i t} + x)$ increases. It follows that $\kappa$ is increasing.

Let $C_t$ be the osculating circle of $\gamma$ at $\gamma(t)$. This is the circle that locally approximates $\gamma$ the best; see e.g. \cite{OscCircles}. A theorem by Tait and later rediscovered by Kneser \cite{tait1896note, kneser1912bemerkungen} states that if a curve has strictly increasing curvature then $t_1 > t_2$ implies that $C_{t_1}$ strictly contains $C_{t_2}$. This implies that for any $t_0 \in [0,1)$ the circle $C_{t_0}$ strictly contains $\gamma((t_0,1))$. Moreover, by the definition of curvature, any circle touching the curve $\gamma$ at $\gamma(t_0)$ of smaller radius must have the property that there are $t>t_0$ with $\gamma(t)$ lying outside of it.

Now let $t_0 \in [0,1)$ be the minimal $t$ for which $|\gamma(t) - 1|$ is maximized. Let $C_M$ be the circle with radius $|\gamma(t_0)-1|$ centered at $1$. The circle $C_M$ touches the curve $\gamma$ at $\gamma(t_0)$ and also contains the the whole curve $\gamma([0,1])$. Therefore its radius is at least that of $C_{t_0}$ and thus $C_M$ contains $C_{t_0}$. It follows that $C_M$ strictly contains the curve $\gamma((t_0,1))$ and thus we conclude that $t_0$ is the unique maximizer of $|\gamma(t) -1|$ for $t \in [0,1]$. See Figure~\ref{fig: example circles} for an example.


It follows that $z_M := - x - re^{\pi i t_0}$ is the unique maximizer of $|z^n - 1|$ for $z$ in the intersection of $D$ with the closed lower half-plane. By symmetry $\overline{z_M}$ is the unique maximizer for $z$ in the intersection of $D$ with the closed upper half-plane. If $t_0 = 0$ then $z_M = \overline{z_M} = z_0$.
\end{proof}

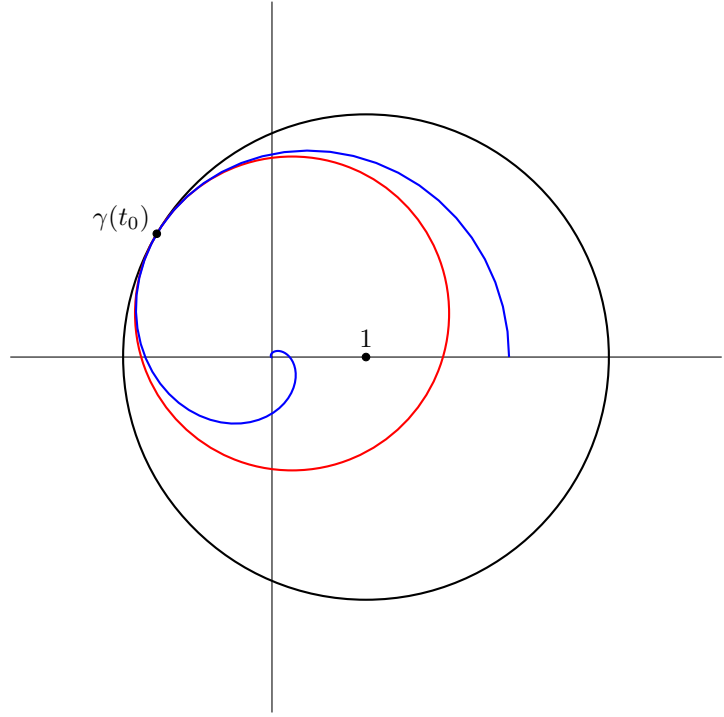
\begin{figure}[ht]
\centering
\label{fig: example circles}
\begin{tikzpicture}
    \draw[] (-3.47898,0) -- (5.97898,0);
    \draw[] (0,-4.72898) -- (0,4.72898);
  \draw[thick] (1.25,0) circle (3.22898);
  \draw[thick, red] (0.267104, 0.579841) circle (2.0878);
  \draw[thick,blue] (3.15203,0)--(3.13149,0.338613)--(3.07025,0.671857)--(2.96942,0.994474)--(2.83083,1.30142)--(2.65699,1.58797)--(2.45102,1.8498)--(2.2166,2.08308)--(1.95788,2.28458)--(1.67939,2.45164)--(1.38596,2.58232)--(1.08257,2.67534)--(0.774308,2.73017)--(0.466224,2.74696)--(0.16324,2.72658)--(-0.129948,2.67053)--(-0.408959,2.58095)--(-0.66981,2.4605)--(-0.908988,2.31236)--(-1.1235,2.14008)--(-1.31094,1.94753)--(-1.46947,1.73883)--(-1.59791,1.5182)--(-1.69567,1.28991)--(-1.76277,1.05819)--(-1.79983,0.827119)--(-1.80798,0.600555)--(-1.78886,0.382074)--(-1.74456,0.174899)--(-1.67753,-0.0181505)--(-1.59049,-0.194698)--(-1.48643,-0.352838)--(-1.36845,-0.491147)--(-1.23973,-0.608694)--(-1.10347,-0.705026)--(-0.962767,-0.780149)--(-0.82059,-0.834502)--(-0.679716,-0.868917)--(-0.542676,-0.884572)--(-0.411717,-0.882942)--(-0.288772,-0.865744)--(-0.175436,-0.834881)--(-0.0729577,-0.792378)--(0.0177651,-0.740332)--(0.0961758,-0.680852)--(0.162046,-0.616009)--(0.215451,-0.547792)--(0.256744,-0.478061)--(0.286519,-0.408519)--(0.305577,-0.340679)--(0.314884,-0.275849)--(0.315533,-0.215114)--(0.308697,-0.159333)--(0.295599,-0.109141)--(0.277468,-0.0649512)--(0.255506,-0.026973)--(0.230862,0.00477481)--(0.204602,0.0304413)--(0.177691,0.0503201)--(0.150979,0.064824)--(0.125185,0.0744584)--(0.100899,0.0797946)--(0.0785758,0.0814447)--(0.0585409,0.0800373)--(0.0409974,0.0761954)--(0.0260363,0.070518)--(0.0136492,0.0635637)--(0.00374326,0.0558377)--(-0.00384369,0.0477832)--(-0.00932657,0.0397746)--(-0.0129575,0.0321154)--(-0.0150104,0.0250383)--(-0.0157675,0.0187075)--(-0.0155068,0.0132243)--(-0.0144917,0.00863347)--(-0.0129631,0.00493144)--(-0.0111326,0.00207506)--(-0.0091791,0)--(-0.0072467,-0.00141641)--(-0.00544447,-0.00225411)--(-0.00384812,-0.00263503)--(-0.00250284,-0.00267)--(-0.00142721,-0.00246256)--(-0.000617806,-0.00210466)--(-0.0000541439,-0.00167379)--(0.000296251,-0.00123144)--(0.000473033,-0.000822717)--(0.000518651,-0.000477012)--(0.000474598,-0.000209584)--(0.000378427,-0.0000237485)--(0.0002616,0.0000864958)--(0.000148195,0.000133652)--(0.0000544373,0.000133935)--(-0.0000110359,0.000104694)--(-0.0000462827,0.0000622163)--(-0.0000547002,0.0000200232)--(-0.0000433631,-0.0000122599)--(-0.0000212384,-0.0000294348)--(0,-0.0000306752)--(0.0000202816,-0.0000189472)--(0.0000267918,0);

    \draw[fill=black] (1.25,0) circle (0.05);
    \node[black] at (1.25,0.25){$1$};
    \draw[fill=black] (-1.53111, 1.64066) circle (0.05);
    \node[black] at (-2, 1.85){$\gamma(t_0)$};
    
\end{tikzpicture}
\caption{An example accompanying the proof of Lemma~\ref{lem: Two maxima}. Here $x = \frac{1}{2}, r=\frac{2}{3}$ and $n=6$. The curve $\gamma$ is drawn in blue, the osculating circle at $\gamma(t_0)$ is drawn in red and the circle touching $\gamma(t_0)$ with center $1$ is drawn in black.}
\end{figure}

\begin{lemma}
    \label{lem: max disk -> indifferent}
    Let $f: \Chat \to \Chat$ be a rational function and $R > 0$ for which
    \begin{enumerate}
        \item $f(\B{R}) \subseteq \B{R}$;
        \item there does not exist any $\tilde{R} > 0$ such that $f(\B{\tilde{R}}) \subseteq \Int(\B{\tilde{R}})$;
        \item for any $z_1, z_2 \in \argmax_{z \in \B{R}} |f(z)|$ we have $|f'(z_1)| = |f'(z_2)|$. 
    \end{enumerate}
    Then $|f'(z)| = 1$ for all $z \in \argmax_{z \in \B{R}} |f(z)|$.
\end{lemma}

\begin{proof}
    Suppose first $|f'(z)| < 1$ for all $z \in \argmax_{z \in B_R} |f(z)|$.
    Note that $\argmax_{z \in B_R} |f(z)|$ is closed and contained in $\partial \B{R}$. It follows that there exist $\delta_1,\delta_2 > 0$ and $\eta < 1$ such that $|f'(z)| < \eta$ for all $z \in S_{\delta_1,\delta_2}$, where 
    \[
        S_{\delta_1,\delta_2} = \{r \cdot e^{i\theta} \cdot z: 1 \leq r < 1+\delta_1, |\theta| < \delta_2 \text{ and } z \in \argmax_{z \in B_R} |f(z)|\}.
    \]
    The set $\partial B_R \setminus S_{\delta_1, \delta_2}$ is a closed subset of $\partial B_R$ whose forward image under $f$ is contained in $\Int(B_R)$. Therefore, by continuity, there exists an $\epsilon \in (0,\delta_1)$ such that the forward image of $\partial \B{(1+\epsilon)R} \setminus S_{\delta_1, \delta_2}$ is contained in $\Int(B_R)$. Now take an element $z \in \partial \B{(1+\epsilon)R} \cap S_{\delta_1, \delta_2}$ and let $\theta_z = \arg(z) $. By construction the whole straight line segment between $R e^{i \theta_z}$ and $z = (1+\epsilon)R e^{i\theta_z}$ lies in $S_{\delta_1,\delta_2}$. Therefore 
    \[
        |f(z)| \leq |f(R e^{i \theta_z})| + |f(z) - f(R e^{i \theta_z})|  < R + \eta |R e^{i \theta_z} - z| = (1 + \eta \epsilon) R.
    \]
    Here we have used that $f(B_R) \subseteq B_R$. We see that $z$ gets mapped into $\Int(\B{(1+\epsilon)R})$ by $f$. We can thus conclude that $f(\B{(1+\epsilon)R}) \subseteq \Int(\B{(1+\epsilon)R})$, which is a contradiction.

    Now suppose that $|f'(z)| > 1$ for all $z \in \argmax_{z \in \B{R}} |f(z)|$. Let $U = \{z: |f'(z)|>1\}$, which is an open neighborhood of $\argmax_{z \in \B{R}} |f(z)|$. Because $f(\B{R}) \subseteq \B{R}$ we can use the Schwarz--Pick theorem which states that for all $z \in \Int(\B{R})$
    \[
        \frac{|f'(z)|}{R^2-|f(z)|^2} \leq \frac{1}{R^2-|z|^2}.
    \]
    It follows that for all $z \in U \cap \Int(\B{R})$ we have $|f(z)| < |z|$. The image of $\partial \B{R} \setminus U$ under $f$ is mapped into $\Int(B_R)$ and thus, by continuity, there is an $\epsilon > 0$ such that $f(\partial \B{(1-\epsilon)R} \setminus U) \subseteq \Int(\B{(1-\epsilon)R})$. It follows that $f(\B{(1-\epsilon)R}) \subset \Int(\B{(1-\epsilon)R})$, which is a contradiction.
\end{proof}

\subsection{Proof of Theorem~\ref{thm: Zero free linear (k)-uniform hypertrees}}

We will now prove Theorem~\ref{thm: Zero free linear (k)-uniform hypertrees}, which we restate for convenience.

\theoremUniformHypertrees*

\begin{proof}
Define 
\[
    g(z) = \left[1 - \left(\frac{z}{1+z}\right)^b \right]^d
\]
so that $f_{\lambda}(z) = \lambda \cdot g(z)$. Consider the set 
\[
    S = \{\lambda \in \mathbb{C}: f_{\lambda}(B_R) \subseteq B_R \text{ for some $R \geq 0$}\}.
\]
For any $\lambda$ with $|\lambda| \leq \left(\max\left\{|g(z)|: z \in \B{1/2}\right\}\right)^{-1}$ we see that $f_{\lambda}(\B{1/2}) \subseteq \B{1/2}$ and thus $S$ is not empty. Moreover, $S$ is closed and if $\lambda \in S$ then $\alpha \cdot \lambda \in S$ for every $\alpha \in \overline{\mathbb{D}}$. If $f_\lambda(B_R) \subseteq B_R$ then $\lambda = f_{\lambda}(0) \in B_R$ and thus $R \geq |\lambda|$. Because $f_{\lambda}(1) = \infty$ it follows that $|\lambda| \le  R < 1$ for all $\lambda \in S$. We can thus conclude that $S = \B{\rho_{d,b}}$. Thus zero-freeness follows from Lemma~\ref{lem: zero-free disk k-uniform hypertrees}.

For the rest let us denote $\rho=\rho_{d,b}$.
Because $\{R \geq 0 : f_{\rho}(B_R) \subseteq B_R\}\subseteq \mathbb{R}$ is closed we can take $R_M$ to be the largest possible radius for which $f_{\rho}(\B{R_M}) \subseteq \B{R_M}$. Clearly there is no $R' > 0$ such that $f_{\rho}(\B{R'}) \subseteq \Int(\B{R'})$ because that would imply that $f_{(1+\epsilon)\rho}(\B{R'}) \subseteq \B{R'}$ for some $\epsilon >0$ contradicting the maximality of $\rho$. Recall that we defined $\mu(z) = z/(1+z)$. We see that $\mu(\B{R_M})$ is a disk containing $\mu(0) = 0$, whose center is $(\mu(R_M) + \mu(-R_M))/2 = - \frac{R_M^2}{1 - R_M^2}$. This center is negative because $R_M < 1$. We have 
\[
    \argmax_{z \in \B{R_M}} |f_{\rho}(z)| = \argmax_{z \in \B{R_M}} |g(z)| = \argmax_{z \in \B{R_M}} |1-\mu(z)^b| = \mu^{-1} \left(\argmax_{w \in \mu(\B{R_M})} |1-w^b|\right).
\]
So, by Lemma~\ref{lem: Two maxima}, there is either a unique $z \in \B{R_M}$ maximizing $|f_\rho(z)|$ or a conjugate pair. In both cases $|f'_\rho(z)|$ is the same for all $z \in \argmax_{z \in \B{R_M}} |f_{\rho}(z)|$. 

Fix $z \in \argmax_{z \in \B{R_M}} |f_{\rho}(z)|$ then, by Lemma~\ref{lem: max disk -> indifferent}, $|f_{\rho}'(z)| = 1$. By the maximum modulus principle $|z| = |f_{\rho}(z)| = R_M$. Let $\lambda = \rho z/f_{\rho}(z)$, then $f_{\lambda}(z) = z$ and $|f_\lambda'(z)| = 1$. Thus it follows from Lemma~\ref{lem: indiff implies zero} that for any $n\ge 1$ integer there is a $\lambda_n\in\mathbb C$ and a $(b+1)$-uniform linear hypertree $\mathcal{T}_n$ with degree at most $d+1$ such that $Z(\mathcal{T}_n;\lambda_n)=0$ and  $|\lambda_n-\lambda|<1/n$. In particular $\lim_{n\to\infty }|\lambda_n|=|\lambda|=\rho$, which proves the maximality part.

Because $z$ is a fixed point on the boundary of a disk that is forward invariant for $f_{\lambda}$, its derivative must be positive real at $z$. Therefore we can conclude that $f_\lambda'(z) = 1$. From Lemma~\ref{lem: fixed points flambda} it follows that $\lambda$ is of the form $\lambda = (1-w^b)^{-d}\cdot\frac{w}{1-w}$ , where $w \neq 1$ is a solution to 
    \[
        -bd \cdot \frac{w^b(1-w)}{1-w^b} = 1.
    \]
Thus, we can conclude that $|\lambda| = \rho$ is at least as large as the smallest in absolute value of such $\lambda$. If $\rho$ were strictly larger than this minimum, then by Lemma~\ref{lem: indiff implies zero} we would conclude that there is a $(b+1)$-uniform tree $\mathcal{T}$ of degree at most $(d+1)$ and $\lambda_1\in \mathbb{C}$ such that $|\lambda_1|<\rho$ and $Z(\mathcal{T};\lambda_1)=0$. But this contradicts the zero-freeness part. Thus $\rho$ is equal to the minimum, which concludes the proof of the formula part.

All that remains is to prove the asymptotics. Fix $b$ and define $h(w) =\frac{w^b(1-w)}{1-w^b}$. Consider the multivariate function $F(w,\eta) = h(\eta w) / \eta^b - 1$. We have $F(w,0) = w^b-1$, which has zeros $\zeta^k$ for $k= 0, \dots, k-1$ and $\zeta = e^{2\pi i /b}$. By the implicit function theorem there are holomorphic functions $\alpha_{k}$ defined in a neighborhood of $\eta = 0$ with $\alpha_k(\eta) = \zeta^k + \mathcal{O}(|\eta|)$ and $F(\alpha_k(\eta),\eta) = 0$. The functions $\alpha_k$ parameterize the solutions to $F(w,\eta) = 0$ for $\eta$ sufficiently small. 

For each $d \geq 2$ let $w_d$ be a solution to $h(w) = -1/bd$ for which $|(1-w^b)^{-d}\cdot\frac{w}{1-w}| = \rho_{d,b}$. Moreover, let $\eta_d = (-1/bd)^{1/b}$, making an arbitrary choice of $b$-th root. Observe that $F(w_d/\eta_d, \eta_d) = 0$ and thus, for sufficiently large $d$, there is a $k_d$ such that $w_d = \eta_d \alpha_{k_d}(\eta_d)$. It follows that 
\[
    w_d^b = -\frac{1}{bd} + \mathcal{O}(d^{-1 - 1/b})
    \quad
    \text{ and }
    \quad 
    |w_d| = \left(bd\right)^{-1/b} + \mathcal{O}(d^{-2/b}).
\]
We conclude that 
\begin{align*}
    \rho_{d,b} = \left|(1-w_d^b)^{-d}\cdot\frac{w_d}{1-w_d}\right| &= \left|1 + \frac{1}{bd} + \mathcal{O}(d^{-1 - 1/b})\right|^{-d} \cdot \left|w_d + \mathcal{O}(w_d^2)\right| \\
    &= \left(e^{-1/b} + \mathcal{O}(d^{-1/b})\right) \cdot \left(\left(bd\right)^{-1/b} + \mathcal{O}(d^{-2/b})\right) \\
    &= (ebd)^{-1/b} + \mathcal{O}(d^{-2/b}).
\end{align*}
\end{proof}


\bibliographystyle{alphaurl}
\bibliography{main}
\end{document}